\documentclass[12pt,twoside,reqno,centertags,
tikz,border=3mm]{amsart}

\usepackage{amsmath,amsthm,amsfonts,amssymb,amsrefs,
ascmac, 
mathrsfs, 
}
\usepackage{enumerate,color}

\usepackage{etoolbox} 

\makeatletter

\@addtoreset{equation}{section}
\makeatother

\theoremstyle{plain} 
\newtheorem{thm}{Theorem}[section]
\newtheorem{lem}[thm]{Lemma}
\newtheorem{cor}[thm]{Corollary}
\newtheorem{prop}[thm]{Proposition}

\theoremstyle{definition} 
\newtheorem{definition}[thm]{Definition}
\newtheorem{rem}[thm]{Remark}
\newtheorem{exa}[thm]{Example}

\newcommand{\R}{\mathbb{R}}

\newcommand{\N}{\mathbb N}

\newcommand{\ep}{\varepsilon}

\newcommand{\Lim}{\displaystyle \lim} 
\newcommand{\Liminf}{\displaystyle \liminf}

\let\tilde=\widetilde

\newcommand{\eqsp}[1]{{\begin{equation}\begin{aligned}#1\end{aligned}\end{
equation}}}

\title[Well-posedness of Hardy-H\'enon equation]{Optimal well-posedness and forward self-similar solution for the Hardy-H\'enon parabolic equation in critical weighted Lebesgue spaces}

\date\today

\author[N. Chikami, M. Ikeda and K. Taniguchi]{Noboru Chikami, Masahiro Ikeda and Koichi Taniguchi}
\address[N. Chikami]
{
Graduate School of Engineering, 
Nagoya Institute of Technology, 
Gokiso-cho, Showa-ku, Nagoya 
466-8555, Japan.}
\email{chikami.noboru@nitech.ac.jp}
\address[M. Ikeda]
{Faculty of Science and Technology,
Keio University, 
3-14-1 Hiyoshi, Kohoku-ku, Yokohama, 223-8522, Japan/ Center for Advanced Intelligence Project
RIKEN, Japan.}
\email{masahiro.ikeda@keio.jp/masahiro.ikeda@riken.jp}
\address[K. Taniguchi]
{Advanced Institute for Materials Research,
Tohoku University,
2-1-1 Katahira, Aoba-ku, Sendai, 980-8577, Japan.}
\email{koichi.taniguchi.b7@tohoku.ac.jp}


\begin{document}

\maketitle

\begin{abstract}
The Cauchy problem for the Hardy-H\'enon parabolic equation is studied in the critical and subcritical regime in weighted Lebesgue spaces on the Euclidean space $\mathbb{R}^d$. 
Well-posedness for singular initial data and existence of non-radial forward self-similar solution of the problem are previously shown only for the Hardy and Fujita cases ($\gamma\le 0$) in earlier works. 
The weighted spaces enable us to treat the potential $|x|^{\gamma}$ as an increase or decrease of the weight, thereby we can prove well-posedness to the problem for all $\gamma$ with $-\min\{2,d\}<\gamma$ including the H\'enon case ($\gamma>0$). 
As a byproduct of the well-posedness, the self-similar solutions to the problem are also constructed for all $\gamma$ without restrictions. 
A non-existence result of local solution for supercritical data is also shown. 
Therefore our critical exponent $s_c$ turns out to be optimal in regards to the solvability. 
\end{abstract} 


\section{Introduction}\label{sec:1}

\subsection{Background and setting of the problem}
We consider the Cauchy problem of the Hardy-H\'enon parabolic equation
\begin{equation}\label{HH}
	\begin{cases}
		\partial_t u - \Delta u = |\cdot|^{\gamma} |u|^{\alpha-1} u,
			&(t,x)\in (0,T)\times D, \\
		u(0) = u_0 \in L^q_{s}(\R^d),
	\end{cases}
\end{equation}
where  $T>0,$ $d\in \mathbb{N}$, $\gamma\in \R,$ $\alpha\in \R,$ 
$D:=\R^d$ if $\gamma\ge0$ and $D:=\R^d\setminus\{0\}$ if $\gamma<0.$ 
Here, $\partial_t:=\partial/\partial t$ is the time derivative, 
$\Delta:=\sum_{j=1}^d\partial^2/\partial x_j^2$ is the Laplace operator on $\R^d$, 
$u=u(t,x)$ is the unknown real- or complex-valued function on 
$(0,T)\times \mathbb R^d$, and $u_0=u_0(x)$ is a prescribed real- or 
complex-valued function on $\mathbb R^d$. 
In this paper, we assume that the initial data $u_0$ belongs to weighted Lebesgue spaces $L^q_s(\R^d)$ given by
\[L^q_s(\R^d):=\left\{ f \in \mathcal{M} (\R^d) \,;\, 
	\|f\|_{L^q_s}
	< \infty \right\}
\] endowed with the norm
\[
	\|f\|_{L^q_s} := \left(\int_{\R^d} ( |x|^s |f(x)|)^q \, dx \right)^\frac1{q},
\]
where $s\in \R$ and $q\in [1,\infty]$ and $\mathcal{M} (\R^d)$ denotes the set of all
Lebesgue measurable functions on $\R^d$. 
We express the time-space-dependent function 
$u$ as $u(t)$ or $u(t,x)$ depending on circumstances. 
We introduce a exponent $\alpha_F(d,\gamma)$ given by
\[
    \alpha_F(d,\gamma):=1+\frac{2+\gamma}{d},
\]
which is often referred as the {\it Fujita exponent} and is known to divide the existence and 
nonexistence of positive global solutions (See \cite[Theorem 1.6]{Qi1998}). 

The equation \eqref{HH} with $\gamma<0$ is known as a 
{\it Hardy parabolic equation} while that with $\gamma>0$ 
is known as a {\it H\'enon parabolic equation}. 
The elliptic part of \eqref{HH}, that is, 
\begin{equation}\nonumber
    -\Delta \phi=|x|^{\gamma}|\phi|^{\alpha-1}\phi,\ \ \ x\in \R^d,
\end{equation}
was proposed by H\'enon as a model to study the rotating 
stellar systems (see \cite{H-1973}), 
and has been extensively studied in the mathematical context, 
especially in the field of nonlinear analysis and variational methods 
(see \cite{GhoMor2013} for example). 
The case $\gamma=0$ corresponds to a heat equation with 
a standard power-type nonlinearity, often called the {\it Fujita equation}, 
which has been extensively studied in various directions. 
Regarding well-posedness of the Fujita equation ($\gamma=0$) in Lebesgue 
spaces, we refer to \cites{Wei1979, Wei1980, Gig86}, among many. 
Concerning the global dynamics and asymptotic behaviors, 
we refer to \cites{Ish2008,IT-arxiv,CIT-arxiv} for the Fujita 
and Hardy cases of \eqref{HH} with Sobolev-critical exponents. 
Articles \cites{HisIsh2018, HisTak-arxiv} give definitive results on 
the optimal singularity of initial data to assure the solvability for $\gamma\le 0.$ 
In \cite{Tay2020}, unconditional uniqueness has been established for the Hardy case $\gamma<0.$ Concerning earlier conditional uniqueness when $\gamma<0$, 
we refer to \cites{BenTayWei2017, Ben2019}. 
Lastly, we refer to \cite{Maj-arxiv} for the analysis of the problem 
\eqref{HH} with an external forcing term in addition to the nonlinear term. 

Let us recall that the equation \eqref{HH} is invariant under the scale transformation 
\begin{equation}\label{scale}
u_{\lambda}(t,x) 
:= \lambda^{\frac{2+\gamma}{\alpha-1}} u(\lambda^2 t, \lambda x), 
	\quad \lambda>0.
\end{equation}
More precisely, if $u$ is the classical solution to \eqref{HH}, 
then $u_{\lambda}$ defined as above also solves the equation 
with the rescaled initial data $ \lambda^{\frac{2+\gamma}{\alpha-1}} u_0(\lambda x).$
Under \eqref{scale}, the $L^q_{s}(\R^d)$-norm scales as follows:
$\|u_\lambda (0)\|_{L^q_{s}} 
	= \lambda^{-s+\frac{2+\gamma}{\alpha-1}-\frac{d}{q}} \|u(0)\|_{L^q_{s}}.$ 
We say that the space $L^q_{s}(\R^d)$ is (scale-){\sl critical} if $s=s_c$ with 
\begin{equation}\label{d:sc}
	s_c=s_c(q)= s_c(d,\gamma,\alpha,q) := \frac{2+\gamma}{\alpha-1} - \frac{d}{q}, 
\end{equation}
{\sl subcritical} if $s<s_c,$ and {\sl supercritical} if $s>s_c.$ 
In particular, when $s=s_c = 0,$ $L^{\frac{d(\alpha-1)}{2+\gamma}}(\R^d)$ 
is a critical Lebesgue space. 

One of our purposes in this article is to establish well-posedness results in the critical and subcritical cases ($s\le s_c$) for all the range of the parameter $\gamma$ such that 
$-\min\{2,d\} < \gamma,$ including the H\'enon case ($\gamma>0$). 
In terms of well-posedness in function spaces containing sign-changing singular data, 
the equation \eqref{HH} has been studied mainly for $\gamma<0$ (Hardy case). 
As far as we know, there has been no result concerning well-posedness in the sense of Hadamard (Existence, uniqueness and continuous dependency) of the H\'enon parabolic equation $\gamma>0$ for sign-changing singular data. 
For the Hardy and Fujita cases that are well-studied, our results provide 
well-posedness in new function spaces (See Remark \ref{r:HH.LWP}). 
We stress that the use of weighted spaces enables us to treat the equations 
for all $\gamma$ in a unified manner. 

Our second purpose of this article is to 
prove the existence of forward self-similar solutions for all of 
Hardy, Fujita and H\'enon cases, without restrictions on the exponent $\alpha.$ 
A forward self-similar solution is a solution such that $u_{\lambda} = u$ for all $\lambda>0,$ 
where $u_{\lambda}$ is as in \eqref{scale}. 
In \cite[Lemma 4.4]{Wan1993}, 
the existence of radially symmetric self-similar solutions for 
$d\ge3$, $\gamma>-2$ and $\alpha\ge1+\frac{2(2+\gamma)}{d-2}$ 
is established. 
Later, the case $\alpha_F(d,\gamma)<\alpha<1+\frac{2(2+\gamma)}{d-2}$ 
is treated in \cite{Hir2008} under some additional restriction on $\gamma,$ 
namely $\gamma\le 0$ for $d\ge4$ and $\gamma\le \sqrt{3}-1$ for $d=3.$   
In \cite[Theorem 1.4]{BenTayWei2017}, the existence of self-similar solutions that are not 
necessarily radially symmetric has been proved for all $\alpha>\alpha_F(d,\gamma),$ 
but only for the Hardy case $\gamma<0$ (See also \cite{Chi2019}). 
Our result (Theorem \ref{t:HH.self.sim}) covers all the previous results and asserts the existence of non-radial forward self-similar solutions for $\gamma$ and $\alpha$ 
such that $-\min(2,d)<\gamma$ and $\alpha > \alpha_F(d,\gamma)$. 

In earlier works, the crux of the matter has been 
the handling of the singular potential $|x|^{\gamma}.$ 
If $\gamma<0$, the conventional methods are to regard the potential 
$|x|^{\gamma}$ as a function belonging either to the Lorentz space $L^{\frac{d}{-\gamma},\infty}(\R^d)$ 
(\cites{BenTayWei2017, Tay2020}) 
or the homogeneous Besov space $\dot B^{\frac{d}{q}+\gamma}_{q,\infty}(\R^d),$ 
$1\le q \le \infty$ (\cite{Chi2019}), 
and apply appropriate versions of H\"older's inequality to establish suitable 
heat kernel estimates. In contrast to their previous works, in this article, 
we treat the potential $|x|^{\gamma}$ as the increase or decrease of the order of 
the weight in $L^q_s(\R^d)$-norms, thereby covering the H\'enon case ($\gamma>0$) as well. 
In this regard, the introduction of the weighted spaces is crucial to our results. 
Indeed, if the data only belongs to the critical Lebesgue space, then we 
may only treat the Hardy case ($\gamma<0$) in our main theorem 
(See Remark \ref{r:HH.LWP} below). 
The proofs of the well-posedness results rely on Banach's fixed point theorem. 
The essential ingredient in the proof of various nonlinear estimates is the following linear estimate for the heat semigroup $\{e^{t\Delta}\}_{t>0}$ on weighted Lebesgue spaces: 
\begin{equation}\nonumber
	\| e^{t\Delta} f\|_{L^q_{s'}} 
	\le C t^{-\frac{d}2 (\frac1{p}-\frac1{q}) - \frac{s-s'}{2} } 	
		\| f\|_{L^p_{s}},
\end{equation}
(see Lemma \ref{l:wLpLq} for precise statement), which is known in the literatures such as \cite{Tsu2011} except for the end-point cases. 
In this article, we first extend the above estimate to the end-point cases $(i)$ $1<p<q=\infty,$ $(ii)$ $p=q=1,$ $(iii)$ $1=p<q<\infty,$ 
$(iv)$ $p=q=\infty$ and $(v)$ $(p,q)=(1,\infty)$. 

To complete the picture of the admissible range of our well-posedness results, we also discuss the non-existence of positive distributional local solutions to \eqref{HH} 
for suitable supercritical data $u_0 \in L^q_{s}(\R^d)$ with $s>s_c.$ 

\subsection{Main results}
In order to state our results, we introduce the following auxiliary function spaces. Let $\mathscr{D}'([0,T)\times\R^d)$ be the space of distributions on $[0,T)\times\R^d$. 
\begin{definition}[Kato class]
\label{def:Kato}
Let $T \in (0,\infty],$ $s\in\R$ and $q\in [1,\infty].$ 
\begin{enumerate}[(1)]
\item In the critical regime, i.e. $\tilde s=s_c$, where $s_c$ is defined by \eqref{d:sc}, for $s<\tilde s$, the space $\mathcal{K}^{s}(T)$ is defined by 
\begin{equation}\nonumber
   \mathcal{K}^{s}(T)
   :=\left\{u\in \mathscr{D}'([0,T)\times\R^d) \,;\, 
   	\|u\|_{\mathcal{K}^{s}(T')}
   <\infty\ \text{for any } T' \in (0,T)\right\}
\end{equation}
endowed with a norm
\[
\|u\|_{\mathcal K^{s}(T)}
	:=\sup_{0\le t\le T}t^{\frac{s_c -s}{2}} \|u(t)\|_{L^q_s}.
\]
We simply write $\mathcal{K}^{s}=\mathcal{K}^{s}(\infty)$ 
when $T=\infty,$ if it does not cause confusion. 
\item In the subcritical regime, i.e. $\tilde s<s_c$, for $s<\tilde s$, the space $\tilde{\mathcal{K}}^{s}(T)$ is defined by 
\begin{equation}\nonumber
   \tilde{\mathcal{K}}^{s}(T)
   :=\left\{u\in \mathscr{D}'([0,T)\times\R^d) \,;\, 
   	\|u\|_{\tilde{\mathcal{K}}^{s}(T')}
   <\infty\ \text{for any } T' \in (0,T)\right\}
\end{equation}
endowed with a norm
\[
\|u\|_{\tilde{\mathcal{K}}^{s}(T)}
	:=\sup_{0\le t\le T}t^{\frac{\tilde s -s}{2}} \|u(t)\|_{L^q_s}.
\]
\end{enumerate}
\end{definition}
For $t\in \R_+$, we introduce the heat kernel $g_t:\R^d\rightarrow \R_+$ given by
\begin{equation}\label{d:h.krnl}
	g_t(x) := (4\pi t)^{-\frac{d}{2}} e^{-\frac{|x|^2}{4t}}, \ x \in\R^d. 
\end{equation}
We denote by $\{e^{t\Delta}\}_{t
\ge 0}$ the free heat semigroup defined by 
\[
(e^{t\Delta} \varphi) (x) := (g_t \ast \varphi) (x)
\]
for $\varphi \in L^1_{loc}(\R^d),$ where $\ast$ denotes the convolution with respect to the space variable. Let $\mathcal{S}'(\R^d)$ denotes the space of the Schwarz distributions.
For $\varphi \in \mathcal{S}'(\R^d)$, $e^{t\Delta}\varphi$ 
is defined by duality. 

In what follows, we denote by $C_0^\infty(\R^d)$ the space of all smooth functions 
with compact support. We also denote by $\mathcal{L}^q_s (\R^d)$ 
the closure of $C_0^\infty(\R^d)$ with respect to the topology of $L^q_s (\R^d).$ 
Next we give a definition of mild solution as follows. 
\begin{definition}[Mild solution]\label{def:sol-A}
Let $T \in (0,\infty]$, $\tilde s\le s_c$ and $u_0 \in L^q_{\tilde s} (\R^d)$. Let $Y := \mathcal{K}^s(T)$ if $\tilde s = s_c$ 
and $Y := \tilde{\mathcal{K}}^{s}(T)$ if $\tilde s<s_c.$
A function $u : [0,T] \times \R^d \to \mathbb C\ \text{or}\ \R$ is called an 
$L^q_{\tilde s} (\R^d)$-mild solution to \eqref{HH} with initial data 
$u(0)=u_0$ if it satisfies 
$u\in C([0,T]; L^q_{\tilde s} (\R^d)) \cap Y$
and the integral equation 
\begin{equation}\label{integral-eq}
u(t,x) = e^{t\Delta} u_0(x) 
	+ \int_0^t e^{(t-\tau)\Delta} 
	\left\{ |\cdot|^{\gamma} |u(\tau,\cdot)|^{\alpha-1}u(\tau, \cdot)\right\}(x) \, d\tau
\end{equation}
for any $t \in [0,T]$ and almost everywhere $x \in \R^d$.
The time $T$ is said to be the maximal existence time, which is denoted 
by $T_m$, if the solution cannot be extended beyond $[0,T).$ 
More precisely, 
\begin{equation}\label{d:Tm}
	T_m = T_m (u_0) := \sup \left\{T>0 \,;\, 
		\left.\begin{aligned}&\text{There exists } \text{ a unique solution $u$ of \eqref{HH}} \\
			&\text{in } C([0,T]; L^q_{\tilde s}(\R^d)) \cap Y 
			\text{ with initial data $u_0$}
		\end{aligned}\right. \right\}.
\end{equation}
We say that $u$ is global in time if $T_m = + \infty$ 
and that $u$ blows up in a finite time otherwise. 
Moreover, we say that $u$ is dissipative if $T_m = + \infty$ and 
\[
	\lim_{t\to\infty} \|u(t)\|_{L^q_{\tilde s}} = 0. 
\]
\end{definition}
The following is one of our main results on local well-posedness of \eqref{HH} in the critical space $L^q_{s_c}(\R^d)$. 
\begin{thm}[Well-posedness in the critical space]	\label{t:HH.LWP}
Let $d\in\mathbb{N},$ $\gamma\in\R$ and $\alpha\in\R$ satisfy 
\begin{equation}\label{t:HH.LWP.c0}
	\gamma> -\min(2,d)
		\quad\text{and}\quad
	\alpha> \alpha_F(d,\gamma).
\end{equation}
Let $q\in [1,\infty]$ be such that 
\begin{equation}\label{t:HH.LWP.c1}
	\alpha\le q \le \infty 
		\quad\text{and}\quad
	\frac1{q} < \min \left\{ \frac{2}{d(\alpha-1)}, \,
		\frac{2}{d(\alpha-1)} + \frac{(d-2)\alpha - d -\gamma}{d(\alpha-1)^2}\right\}
\end{equation}
and let $s \in \R$ be such that 
\begin{equation}\label{t:HH.LWP.c2}
	s_c - \frac{d(\alpha-1)}{\alpha} \left(\frac{2}{d(\alpha-1)} - \frac1{q} \right) \le s
	< \min \left\{ s_c, \, s_c + \frac{(d-2)\alpha - d -\gamma}{\alpha(\alpha-1)} \right\}. 
\end{equation}
Then the Cauchy problem \eqref{HH} is locally well-posed in $L^q_{s_c}(\R^d)$ for arbitrary data $u_0\in L^q_{s_c}(\R^d)$ and globally well-posed for small data $u_0\in L^q_{s_c}(\R^d)$. 
More precisely, the following assertions hold. 
\begin{enumerate}[$(i)$]
\item {\rm (}Existence{\rm )} 
For any $u_0 \in L^q_{s_c}(\R^d)$ with $q <\infty$ 
(Replace $L^\infty_{s_c}(\R^d)$ with 
$\mathcal{L}^\infty_{s_c}(\R^d)$ when $q = \infty$), 
there exist a positive number $T$ and an $L^q_{s_c}(\R^d)$-mild solution 
$u$
\ to \eqref{HH} satisfying 
\begin{equation}\label{t:HH.LWP.est} 
	\|u\|_{\mathcal{K}^s(T)} 
		\le 2 \|e^{t\Delta} u_0 \|_{\mathcal{K}^s(T)}. 
\end{equation}
Moreover, the solution can be extended to the maximal interval $[0,T_m),$ 
where $T_m$ is defined by \eqref{d:Tm}. 
\item {\rm (}Uniqueness{\rm )} 
Let $T>0.$ If $u, v \in \mathcal{K}^s(T)$ satisfy 
\eqref{integral-eq} with $u(0) = v(0)=u_0 \in L^q_{s_c}(\R^d)$ 
(Replace $L^\infty_{s_c}(\R^d)$ with $\mathcal{L}^\infty_{s_c}(\R^d)$ when $q=\infty$), 
then $u=v$ on $[0,T].$
\item {\rm (}Continuous dependence on initial data{\rm )}  
Let $u$ and $v$ be the $L^q_{s_c}(\R^d)$-mild solutions constructed in (i) with given initial data $u_0$ and $v_0$ respectively. 
Let $T(u_0)$ and $T(v_0)$ be the corresponding existence times. 
Then there exists a constant $C$ depending on $u_0$ and $v_0$ such that 
the solutions $u$ and $v$ satisfy 
\begin{equation}\nonumber	
	\|u-v\|_{L^\infty(0,T;L^q_{s_c}) \cap \mathcal{K}^s(T)} 
	\le C \|u_0-v_0\|_{L^q_{s_c}}
\end{equation}
for some $T\le \min\{T(u_0), T(v_0)\}.$ 
\item {\rm (}Blow-up criterion{\rm )} 
If $u$ is an $L^q_{s_c}(\R^d)$-mild solution constructed in the assertion $(i)$ and 
$T_m<\infty,$ then $\|u\|_{\mathcal{K}^s(T_m)}=\infty.$
\item {\rm (}Small data global existence and dissipation{\rm )} 
There exists $\ep_0>0$ depending only on $d,\gamma,\alpha,q$ and $s$ such that if $u_0 \in \mathcal{S}'(\R^d)$ satisfies 
$\|e^{t\Delta}u_0\|_{\mathcal{K}^s}<\ep_0,$ then $T_m=\infty$ and 
$\|u\|_{\mathcal{K}^s} \le 2\ep_0.$ Moreover, the solution $u$ is dissipative. In particular, if $\|u_0\|_{L^p_{s_c}}$ is sufficiently small, then 
$\|e^{t\Delta}u_0\|_{\mathcal{K}^s}<\ep_0.$ 
\end{enumerate}
\end{thm}
\begin{rem}[Optimality of the power $\alpha$ for the nonlinearity]
By the blow-up result in \cite{Qi1998}, 
the condition $\alpha>\alpha_F(d,\gamma)$ is known to be optimal. 
Indeed, if $\alpha \le \alpha_F(d,\gamma)$, then the solutions of \eqref{HH}
with positive initial data blow up in a finite time. 
\end{rem}
\begin{rem}[Uniqueness $(ii)$]
In $(ii)$, $T$ is arbitrary and there is no restriction on the size of the quantity $\|u\|_{\mathcal{K}^s(T)}.$ We note that this uniqueness result concerns a so-called {\sl conditional} uniqueness since we can prove that $u\in \mathcal{K}^s(T)$ is a solution to \eqref{HH} if and only if 
$u \in C([0,T] ; L^q_{s_c}(\R^d)) \cap \mathcal{K}^s(T)$ is a solution to \eqref{HH}, 
provided that $u_0 \in L^q_{s_c}(\R^d).$ See Remark \ref{r:crt.est2} below. 
We note that for the Hardy case, unconditional uniqueness has been established by \cite{Tay2020} in the Lebesgue framework. 
\end{rem}
\begin{exa}[Small data global existence $(v)$]
We give a typical example of the initial data $u_0$ 
satisfying the assumptions in $(v)$ : 
$u_0 \in L^1_{loc}(\R^d)$ such that 
$|u_0(x)| \le c |x|^{-\frac{2+\gamma}{\alpha-1}}$ for almost all $x\in\R^d,$ 
where $c$ is a sufficiently small constant. 
This initial data in particular generates a self-similar solution. 
See Theorem \ref{t:HH.self.sim} below. 
\end{exa}
\begin{rem}[New contributions for $\gamma\neq0$]
\label{r:HH.LWP}
For the Hardy case $\gamma > 0,$ Theorem \ref{t:HH.LWP} is new 
concerning sign-changing solutions for singular initial data. 
Theorem \ref{t:HH.LWP} also gives a new result in the Hardy case ($\gamma<0$). 
In particular, when $s_c \equiv 0$, that is, $q = \frac{d(\alpha-1)}{2+\gamma}$, the critical space is the usual Lebesgue space 
$L^{\frac{d(\alpha-1)}{2+\gamma}}(\R^d).$ 
Theorem \ref{t:HH.LWP} gives a new well-posedness result in the usual Lebesgue space $L^{\frac{d(\alpha-1)}{2+\gamma}}(\R^d)$ for $d\ge2$ and $-2<\gamma<0$. 
\end{rem}
\begin{rem}
We note that $s_c$ is always positive when $\gamma>0$ 
while $s_c$ can be either negative or non-negative. In other words, 
the initial data $u_0$ must have a stronger decay at infinity when $\gamma>0.$ 
\end{rem}

We next discuss global existence of forward self-similar solutions to \eqref{HH}. 
As mentioned earlier, the result below is not known in the literature for large $\gamma>0.$ 
\begin{thm}[Existence of forward self-similar solutions]	\label{t:HH.self.sim}
Let $d\in\mathbb{N},$ $\gamma\in\R$ and $\alpha\in\R$ satisfy \eqref{t:HH.LWP.c0}. 
Let $\varphi(x) := \omega(x) |x|^{-\frac{2+\gamma}{\alpha-1}},$ 
where $\omega\in L^\infty(\R^d)$ is homogeneous of degree 0 and 
$\|\omega\|_{L^\infty}$ is sufficiently small so that $\|e^{t\Delta}\varphi\|_{\mathcal{K}^s}<\varepsilon_0$, where $\varepsilon_0$ appears in Theorem \ref{t:HH.LWP}.
Then there exists a self-similar solution $u_\mathcal{S}$ of 
\eqref{HH} with the initial data $\varphi$ such that $u_\mathcal{S}(t) \to \varphi$ in $\mathcal{S}'(\R^d)$ as $t\to0.$ 
\end{thm}

The following theorem deals with the local well-posedness 
of \eqref{HH} in the subcritical space $L^q_{\tilde s}(\R^d)$ with $\tilde s< s_c.$ 
\begin{thm}[Well-posedness in the subcritical space]	\label{t:HH.LWP.sub}
Let $d\in\mathbb{N},$ $\gamma\in\R$ and $\alpha\in\R$ satisfy \eqref{t:HH.LWP.c0}. 
Let $\tilde s\in\R$ be such that 
\begin{equation}\label{t:HH.LWP.sub.cs}
	\max\left\{-\frac{d}{\alpha}, \, \frac{\gamma}{\alpha-1} \right\}
		<\tilde s < \frac{2+\gamma}{\alpha-1}.
\end{equation}
Let $q\in[1,\infty]$ be such that 
\begin{equation}\label{t:HH.LWP.sub.c1}
	\alpha\le q \le \infty 
		\quad\text{and}\quad
	-\frac{\tilde s}{d} < \frac1{q} < \min \left\{ \frac{2}{d(\alpha-1)}, \,
		\frac1{\alpha} \left(1-\frac{\tilde s}{d} \right), \, 
		\frac1{d} \left(\frac{2 + \gamma}{\alpha-1} -\tilde s \right)  \right\}
\end{equation}
and let $s \in \R$ be such that 
\begin{equation}\label{t:HH.LWP.sub.c2}
	\frac{\tilde s+\gamma}{\alpha} \le s
		\quad\text{and}\quad
	- \frac{d}{q} < s 
	< \min \left\{ \frac{d+\gamma}{\alpha} - \frac{d}{q}, \tilde s \right\}. 
\end{equation}
Then the Cauchy problem \eqref{HH} is locally well-posed in $L^q_{\tilde{s}}(\R^d)$ for arbitrary data $u_0\in L^q_{\tilde{s}}(\R^d)$. More precisely, the following assertions hold. 
\begin{enumerate}[$(i)$]
\item {\rm (}Existence{\rm )} 
For any $u_0 \in L^q_{\tilde s}(\R^d),$ 
there exist a positive number $T$ depending only on $\|u_0\|_{L^q_{\tilde s}}$ and an $L^q_{\tilde s}(\R^d)$-mild solution 
$u $
\ to \eqref{HH} satisfying 
\begin{equation}\nonumber	
	\|u\|_{\tilde{\mathcal{K}}^s(T)} 
		\le 2 \|e^{t\Delta} u_0 \|_{\tilde{\mathcal{K}}^s(T)}. 
\end{equation}
Moreover, the solution can be extended to the maximal interval 
$[0,T_m),$ where $T_m$ is defined by \eqref{d:Tm}.
\item {\rm (}Uniqueness in $\tilde{\mathcal{K}}^s(T)${\rm )} 
Let $T>0.$ If $u, v \in \tilde{\mathcal{K}}^s(T)$ satisfy 
\eqref{integral-eq} with $u(0) = v(0)=u_0,$ then $u=v$ on $[0,T].$ 
\item {\rm (}Continuous dependence on initial data{\rm )}  
For any initial data $u_0$ and $v_0$ in $L^q_{\tilde s}(\R^d),$ 
let $T(u_0)$ and $T(v_0)$ be the corresponding existence time given by $(i).$ 
Then there exists a constant $C$ depending on $u_0$ and $v_0$ such that 
the corresponding solutions $u$ and $v$ satisfy 
\begin{equation}\nonumber	
	\|u-v\|_{L^\infty(0,T;L^q_{\tilde s}) \cap \tilde{\mathcal{K}}^s(T)} 
	\le C \|u_0-v_0\|_{L^q_{\tilde s}}
\end{equation}
for some $T\le \min\{T(u_0), T(v_0)\}.$ 
\item {\rm (}Blow-up criterion{\rm )} If $T_m<\infty,$ 
	then $\lim_{t\rightarrow T_m-0}\|u(t)\|_{L^q_{\tilde s}}=\infty.$ 
Moreover, the following lower bound of blow-up rate holds: 
there exists a positive constant $C$ independent of $t$ such that 
\begin{equation}\label{t:HH.LWP:Tm}
	\|u(t)\|_{L^q_{\tilde s}} \ge  \frac{C}{(T_m - t)^{\frac{s_c-\tilde s}{2}} }
\end{equation}
for $t\in (0,T_m)$.
\end{enumerate}
\end{thm}
\begin{rem}
Note that 
\eqref{t:HH.LWP.sub.c1} implies $\tilde s<s_c,$ i.e., $u_0 \in L^q_{\tilde s}(\R^d)$ 
is a scale-subcritical data.  
\end{rem}

Finally, for the scale-supercritical case, i.e. $s>s_c$, we prove non-existence of a weak local positive solution, whose definition is given below. More precisely, we may prove that there exists a positive initial data $u_0$ in $L^q_s(\R^d)$ with $s>s_c$ that does not generate a 
local solution to \eqref{HH} even in the distributional sense. 
\begin{definition}[Weak solution]
\label{d:w.sol}
Let $T>0$. We call a function $u:[0,T)\times \R^d\rightarrow \R$ a weak solution to the Cauchy problem \eqref{HH} 
if $u$ belongs to $L^{\alpha}(0,T;L^{\alpha}_{\frac{\gamma}{\alpha},loc}(\R^d))$ 
and if it satisfies the equation \eqref{HH} in the distributional sense, i.e., 
\begin{align}\label{weak}
\notag\int_{\R^d} &u(T',x) \eta (T',x) \, dx-\int_{\R^d} u_0(x) \eta (0,x) \, dx\\
	&= \int_{[0,T']\times\R^d} u(t ,x)(\Delta \eta + \eta_t) (t ,x) 
	+ |x|^{\gamma} |u(t, x)|^{\alpha-1} u(t,x) \,\eta(t,x)  \, dx\,dt
\end{align}
for all $T'\in [0,T]$ and for all $\eta \in C^{1,2}([0,T]\times \R^d)$ such that 
$\operatorname{supp} \eta(t, \cdot)$ is compact. 
\end{definition}
We remark that our $L^q_{\tilde s}(\R^d)$-mild solutions 
are weak solutions in the above sense. See Lemma \ref{mildweak} in Appendix.
\begin{thm}[Nonexistence of local positive weak solution]
\label{t:nonex}
Let $d\in \mathbb N$ and $\gamma \in \mathbb R$. 
Assume that $q\in [1,\infty],$ $\alpha\in\R$ and $s\in\R$ satisfy 
$\alpha>\max(1, \alpha_F(d,\gamma))$ and $s>s_c$. 
Then there exists an initial data $u_0 \in L^q_s (\R^d)$ such that the 
problem \eqref{HH} with $u(0)=u_0$ has no local positive weak solution. 
\end{thm}
\bigbreak
The rest of the paper is organized as follows: In Section 2, we 
prove the linear estimates and nonlinear ones in weighted Lebesgue spaces. 
Section 3 is devoted to the proof of Theorems \ref{t:HH.LWP}, \ref{t:HH.LWP.sub} 
and \ref{t:HH.self.sim}. 
We then give a sketch of the proof of Theorem \ref{t:nonex} in Section 5. 
In Appendix, we collect some elementary properties related to our function spaces and prove Lemma \ref{mildweak}. 

\section{Linear and nonlinear estimates}
Throughout the rest of the paper, we denote by $C$ 
a harmless constant that may change from line to line. 
\subsection{Linear estimate}
The following estimate for the heat semigroup $\{e^{t\Delta}\}_{t\ge0}$ in weighted Lebesgue space is known except for the endpoint cases (see \cites{Tsu2011, OkaTsu2016}).
\begin{lem}[Linear estimate]
	\label{l:wLpLq}
Let $d\in\N,$ $1\le p \le q \le \infty$ and  
\begin{equation}	\label{l:wLpLq:cs}
	-\frac{d}{q} < s' \le s < d\left( 1-\frac{1}{p} \right).
\end{equation}
In addition, $s\le 0$ when $p=1$ and $0\le s'$ when $q=\infty.$ 
In particular,  \eqref{l:wLpLq:cs} is understood as $s'=s=0$ when $p=1$ and $q=\infty.$ 
Then there exists some positive constant $C$ depending on $d,$ $p,$ $q,$ $s$ and
$s'$ such that 
\begin{equation}\nonumber
	\| e^{t\Delta} f\|_{L^q_{s'}} 
	\le C t^{-\frac{d}2 (\frac1{p}-\frac1{q}) - \frac{s-s'}{2} } 
		\| f\|_{L^p_{s}}
\end{equation}
for all $f\in L^p_{s}(\R^d)$ and $t>0$. 
Moreover, condition \eqref{l:wLpLq:cs} is optimal. 
\end{lem}
We mainly focus on the endpoint cases in the following proof.
\begin{proof}
The inequality for $1< p\le q <\infty$ follows from Lemma 3.2, \cite[Proposition C.1]{OkaTsu2016} and the fact that the weight function 
$|x|^{s p}$ belongs to the Muckenhoupt class $A_p$ if and only if 
$- \frac{d}{p} < s < d(1- \frac1{p}).$ 

For the endpoint exponents, we divide the proof into five cases : 
$(i)$ $1<p<q=\infty,$ $(ii)$ $p=q=1,$ $(iii)$ $1=p<q<\infty,$ 
$(iv)$ $p=q=\infty$ and $(v)$ $(p,q)=(1,\infty).$ 
It suffices to prove the inequality for $e^{\Delta} f$ and then resort to 
a dilation argument as in the proof of \cite[Proposition 2.1]{BenTayWei2017}. 

Throughout the proof of this lemma, 
we write $a\lesssim b$ if $a \le C b$ with some constant $C.$\\ 
\underline{$(i)$ $1<p<q=\infty$}: Since $|x|^{s'} \lesssim |x-y|^{s'} + |y|^{s'}$ if $s' \ge 0,$ we have 
\begin{equation*}
	|x|^{s'} | e^{\Delta} f(x)| 
	\lesssim \int_{\R^d} |x-y|^{s'} g(x-y) |f(x)| \, dy 
		+ \int_{\R^d} |y|^{s'} g(x-y) |f(x)| \, dy = : I_1 + I_2.
\end{equation*}
For $I_1,$ H\"older's inequality with $\frac1{p}+\frac1{p'}=1,$ $p>1,$ leads to 
\begin{equation*}\nonumber
	I_1 
		 \le \left( \int_{\R^d} (|y|^{-s} |x-y|^{s'} g(x-y) )^{p'}\, dy \right)^{\frac1{p'}}  \, \|f\|_{L^p_s} 
		\lesssim \|f\|_{L^p_s}, 
\end{equation*}
thanks to Lemma \ref{l:g.unfrm.bnd} $(1)$ with $q\equiv p',$ $a\equiv s$ 
and $b\equiv s',$ where $0\le s<\frac{d}{p'}$ and $s'\ge0.$ 
Similarly, H\"older's inequality and Lemma \ref{l:g.unfrm.bnd} $(2)$ 
with $q\equiv p'$ and $c\equiv s-s'$ yields 
\begin{align*}\nonumber
	I_2 
		& \le \left( \int_{\R^d} (|y|^{-(s-s')} g(x-y) )^{p'}\, dy \right)^{\frac1{p'}}  \, \|f\|_{L^p_s} 
		\lesssim \|f\|_{L^p_s},
\end{align*}
where $0\le s-s' < \frac{d}{p'}.$ 
Thus, $\| e^{\Delta} f\|_{L^{\infty}_{s'}} \lesssim \| f\|_{L^p_{s}}$ provided that 
$0 \le s' \le s < d\left(1-\frac1{p}\right).$ 

\underline{$(ii)$ $p=q=1$}: We have 
$|y|^{-s} \lesssim |x-y|^{-s} + |x|^{-s}$ if $s\le 0$ and thus 
\begin{align*}
	\|e^{\Delta} f\|_{L^1_{s'}} 
	&\lesssim \int_{\R^d} |x|^{s'} \int_{\R^d}  g(x-y) |x-y|^{-s} |y|^{s} |f(y)| \,dy \, dx \\
		&\qquad\qquad+ \int_{\R^d} |x|^{s'-s} \int_{\R^d}  g(x-y) |y|^{s} |f(y)| \,dy \, dx \\ 
	&\lesssim \int_{\R^d} \left( \int_{\R^d} |x|^{s'} g(x-y) |x-y|^{-s} \, dx \right) |y|^{s} |f(y)|  \,dy\\
		&\qquad\qquad+ \int_{\R^d} \left( \int_{\R^d}  |x|^{s'-s}  g(x-y) \, dx\right) |y|^{s} |f(y)| \,dy \ \lesssim \|f\|_{L^1_s}
\end{align*}
thanks to Fubini's theorem and Lemma \ref{l:g.unfrm.bnd} with 
$q\equiv 1,$ $a\equiv -s',$ $b\equiv -s$ and $c\equiv s-s',$ where 
$0 \le -s' < d,$ $0 \le -s$ and $0 \le s-s' < d.$ 
Thus, $\| e^{\Delta} f\|_{L^{1}_{s'}} \lesssim \| f\|_{L^1_{s}}$ provided that $-d<s'\le s \le 0.$

\underline{$(iii)$ $1=p < q < \infty$}: By H\"older's inequality 
with $1=\frac1{q}+\frac1{q'},$ $q<\infty,$ we have  
\begin{align*}
	|e^{\Delta} f(x)| 
	&\le \left( \int_{\R^d} |y|^{-sq} g(x-y)^q |y|^s |f(y)| \, dy \right)^{\frac1{q}} \|f\|_{L^1_s}^{\frac1{q'}}
\end{align*}
for $s\le 0.$ Taking the $L^q_{s'}(\R^d)$-norm of the both sides of the above, we obtain 
\begin{align*}
	\|e^{\Delta} f\|_{L^q_{s'}} 
	&\le \left( \int_{\R^d} |x|^{s'q} \left( \int_{\R^d} |y|^{s(1-q)} g(x-y)^q |f(y)| \, dy \right) dx \right)^{\frac1{q}} \|f\|_{L^1_s}^{\frac1{q'}}.
\end{align*}
Since $|y|^{-qs} \lesssim |x-y|^{-qs} + |x|^{-qs}$ if $s<0,$ 
Fubini's theorem and Lemma \ref{l:g.unfrm.bnd} with 
$q\equiv q,$ $a\equiv -s',$ $b\equiv -s$ and $c\equiv s-s'$ yield  
\begin{align*}
\int_{\R^d} |x|^{s'q} 
	&\left( \int_{\R^d} |y|^{-sq} g(x-y)^q |y|^s |f(y)| \, dy \right) dx \\
	&\lesssim \int_{\R^d} |x|^{s'q} \left( \int_{\R^d} |x-y|^{-qs} g(x-y)^q |y|^{s} |f(y)| \, dy \right) dx \\
	&\qquad\qquad 
	+ \int_{\R^d} |x|^{-(s-s')q} \left( \int_{\R^d} g(x-y)^q |y|^{s} |f(y)| \, dy \right) dx \\
	&\lesssim \int_{\R^d} |y|^{s} |f(y)| \left( \int_{\R^d} (|x|^{s'} |x-y|^{-s} g(x-y) )^q \, dx \right) dy \\
	&\qquad\qquad 
	+ \int_{\R^d} |y|^{s} |f(y)| \left( \int_{\R^d} (|x|^{-(s-s')} g(x-y))^q \, dx \right) dy
\lesssim \|f\|_{L^1_s}, 
\end{align*}
where $0\le -s' < \frac{d}{q},$ $0\le -s$ and $0 \le s-s' < \frac{d}{q}.$ 
Thus, $\| e^{\Delta} f\|_{L^{q}_{s'}} \lesssim \| f\|_{L^1_{s}}$ provided that 
$-\frac{d}{q}<s'\le s \le 0.$

\underline{$(iv)$ $p=q=\infty$}: 
Since $|x|^{s'} \lesssim |x-y|^{s'} + |y|^{s'}$ if $s' \ge 0,$ we have 
\begin{align*}
	|x|^{s'} &|e^{\Delta} f(x)| 
	\le |x|^{s'} \int_{\R^d} |y|^{-s} g(x-y) \, dy \|f\|_{L^\infty_s}\\
	&\lesssim \left(  \int_{\R^d} |y|^{-s} |x-y|^{s'} g(x-y) \, dy 
		+ \int_{\R^d} |y|^{s'-s} g(x-y) \, dy \right)  \|f\|_{L^\infty_s}
	\lesssim \|f\|_{L^\infty_s}
\end{align*}
thanks to Lemma \ref{l:g.unfrm.bnd} with 
$q\equiv 1,$ $a\equiv s,$ $b\equiv s'$ and $c\equiv s-s',$ where 
$0\le s< d,$ $0\le s'$ and $0\le s-s' < d.$ 
Thus, $\| e^{\Delta} f\|_{L^{\infty}_{s'}} \lesssim \| f\|_{L^{\infty}_{s}}$ provided that 
$0\le s'\le s <d.$
The case $(v)$ $(p,q) = (1,\infty)$ is trivial. 
We complete the proof of the endpoint estimates. 

\smallbreak
Next, we prove the optimality of \eqref{l:wLpLq:cs} for 
$1< p \le q < \infty$ by contradiction. 
Suppose that the inequality holds when $s' \le -\frac{d}{q}.$ 
We notice that every function $g$ in $L^q_{s'}(\R^d)$ must satisfy 
$\displaystyle \liminf_{|x|\to0} |x|^{\frac{d}{q}+s'} |g(x)| =0$ thanks to 
Corollary \ref{c:wLp.sg.dcy} in Appendix. 
In particular, we have 
$
	\Liminf_{|x|\to0} |g(x)| =0
$
as $0 \le -s'-\frac{d}{q}.$ Since $0<\frac{d}{p}+s < d,$ a function $f$ defined by 
\[
f(x) := \left\{\begin{aligned}
	&C, &&|x|\le 1\\
	&0, &&\text{else},
\end{aligned}\right. 
\]
where $C$ is a positive constant, belongs to $L^p_s(\R^d).$ However, clearly 
\[
\displaystyle \liminf_{|x|\to0} |e^{t\Delta} f(x)| \neq 0,
\]
which implies that $e^{t\Delta} f \notin L^{q}_{s'}(\R^d)$ and leads to a contradiction. 

The optimality of the upper bound of \eqref{l:wLpLq:cs} is based on the 
fact that the space $L^p_s(\R^d)$ contains functions that are not in $L^1_{loc}(\R^d),$ 
if $d\left( 1-\frac{1}{p} \right)< s.$ Let 
\[
f(x) := \left\{\begin{aligned}
	&|x|^{-d}, &&|x|\le 1\\
	&0, &&\text{else},
\end{aligned}\right. 
\]
so that it belongs to $L^p_s(\R^d)$ as $p(d-s)<d$ 
(Note that the space $L^p_s(\R^d)$ is defined for all measurable functions). 
A standard argument then shows that $e^{t\Delta} f$ does not make sense for the 
function $f.$ Indeed, for every $t>0$ and every $x$ such that $|x|\le 1,$ the estimates hold:
\begin{equation*}
	e^{t\Delta} f(x) 
	=  (4\pi t)^{-\frac{d}2} \int_{|y|\le 1} e^{\frac{-|x-y|^2}{4t}} |y|^{-d} \, dy 
	\ge  (4\pi t)^{-\frac{d}2} \int_{|y|\le 1} e^{-\frac{1}{t}} |y|^{-d} \, dy = \infty, 
\end{equation*}
where we have used $|x-y|\le 2.$ 
Thus $e^{t\Delta}$ is not well-defined in $L^p_s(\R^d)$ if $d\left( 1-\frac{1}{p} \right)< s.$ 
When $s= d(1-\frac1{p}),$ it suffices to take 
\[
f(x) := \left\{\begin{aligned}
	&|x|^{-d} 
	\left( \log\left(e+ \frac1{|x|}\right) \right)^{-\frac{a}{p}}, &&|x|\le 1,\\
	&0, &&\text{else},
\end{aligned}\right. 
\]
where $p\ge a>1,$ and show that $e^{t\Delta} f$ is not well-defined for the function 
by carrying out the same argument as above. Thus, we conclude the lemma. 
\end{proof}

\subsection{Nonlinear estimates}
Given $u_0\in L^q_{s_c}(\R^d)$ in the critical regime (resp. $L^q_{\tilde{s}}(\R^d)$ in the subcritical regime) and $T>0,$ let us define a map 
$\Phi : u \mapsto \Phi(u)$ on $\mathcal{K}^s(T)$ (resp. $\tilde{\mathcal{K}}^s(T)$) by 
\begin{equation}\label{map}
	\Phi(u) (t) := e^{t\Delta} u_0 + N(u)(t)
\end{equation}
with 
\begin{equation}\label{mapN}
	N(u)(t) :=  \int_0^t e^{(t-\tau)\Delta} 
	\left\{ |\cdot|^{\gamma} F(u(\tau,\cdot)) \right\} d\tau
	\quad\text{and}\quad
	F(u) := |u|^{\alpha-1}u.
\end{equation}

\subsubsection{Critical case}
The following are the stability and contraction estimates in the critical regime. 
The assertion $(2)$ below for $\theta<1$ is not required in the proof of existence 
but is used in the proof of uniqueness.  
\begin{lem}
\label{l:Kato.est}
Let $T \in (0,\infty]$ and $d\in\mathbb{N}.$ 
Let $\gamma\in\R$ and $\alpha\in\R$ satisfy \eqref{t:HH.LWP.c0}. 
\begin{enumerate}[$(1)$]
\item 
Let $q\in [1,\infty]$ be such that 
\begin{equation}\label{l:Kato.est.c1}
	\alpha\le q \le \infty \quad\text{and}\quad
	\frac1{q} < \min \left\{ \frac{2}{d(\alpha-1)}, \, 
	\frac{2}{d(\alpha-1)} + \frac{(d-2)\alpha - d -\gamma}{d\alpha (\alpha-1)} \right\}. 
\end{equation}
Let $s \in \R$ be such that 
\begin{equation}\label{l:Kato.est.c2}
	\frac{\gamma}{\alpha-1}\le s
		\quad\text{and}\quad
	\max\left\{- \frac{d}{q}, \, s_c - \frac2{\alpha} \right\} < s 
	< \min \left\{ s_c, \, s_c + \frac{(d-2)\alpha - d -\gamma}{\alpha(\alpha-1)} \right\},
\end{equation}
where $s_c$ is as in \eqref{d:sc}. 
Then there exists a positive constant $C_0$ 
depending only on $d,$ $\alpha,$ $\gamma,$ $q$ and $s$ such that 
the map $N$ defined by \eqref{mapN} satisfies 
\begin{equation}\label{l:Kato.est1}
\|N(u)\|_{\mathcal{K}^s(T)} 
	\le C_0 \|u\|_{\mathcal{K}^s(T)}^{\alpha}
\end{equation}
for all $u \in \mathcal{K}^s(T).$ 
\item
Let $q\in [1,\infty]$ be such that 
\begin{equation}\label{l:Kato.est.c1'}
\begin{aligned}
	&\alpha \le q \le \infty, \\
		\text{and}\quad
	&\frac1{q} < \min \left\{ \frac2{d(\alpha-1)},\, 
	\frac2{d(\alpha-1)} + \frac{\theta(d-2)(\alpha-1) - 2 - \gamma}
			{d(\alpha-1)(1+\theta(\alpha-1))} \right\},
\end{aligned}
\end{equation}
where $\theta \in (0,1]$ ($\frac1{2+\gamma} < \theta$ if d=1). 
Let $s \in \R$ be such that 
\begin{equation}\label{l:Kato.est.c2'}
\begin{aligned}
&	s_c - \frac{d}{\theta} \left( \frac{2}{d(\alpha-1)} -\frac{1}{q} \right)\le s \\
		\text{and}\quad
&	\max\left\{ -\frac{d}{q}, \, s_c - \frac2{1+\theta(\alpha-1)}  \right\}  
	< s < \min\left\{ s_c, \, s_c + \frac{(d-2)\alpha-d-2}{(1+\theta(\alpha-1))(\alpha-1)}
					 \right\}. 
\end{aligned}
\end{equation}
Then there exists a positive constant $C_1$ 
depending only on $d,$ $\alpha,$ $\gamma,$ $q,$ $s$ and $\theta$ such that 
the map $N$ defined by \eqref{mapN} satisfies  
\begin{equation}\label{l:Kato.est2}
\begin{aligned}
\|N(u) - N(v)\|_{\mathcal{K}^s(T)} 
	\le C_1 &\left(\|u\|_{\mathcal{K}^s(T)}
				+\|v\|_{\mathcal{K}^s(T)} \right)^{\theta(\alpha-1)} \\
		&\times\left(\|u\|_{L^\infty(0,T; L^q_{s_c}) }
				+\|v\|_{L^\infty(0,T; L^q_{s_c}) } \right)^{(1-\theta)(\alpha-1)} 
	\|u-v\|_{\mathcal{K}^s(T)}
\end{aligned}
\end{equation}
for all $u,v \in \mathcal{K}^s(T) \cap L^\infty(0,T ; L^q_{s_c}(\R^d))$ 
($u,v \in \mathcal{K}^s(T)$ if $\theta = 1$). 
\end{enumerate}
\end{lem}
\begin{rem}
Note that \eqref{l:Kato.est.c1'} and \eqref{l:Kato.est.c2'} for $\theta=1$ 
are equivalent to \eqref{l:Kato.est.c1} and \eqref{l:Kato.est.c2}, respectively. 
The estimate \eqref{l:Kato.est2} fails for $\theta=0$ as $C_1$ is divergent as 
$\theta\to0.$ 
\end{rem}
\begin{proof}
We first prove \eqref{l:Kato.est1}. We have 
\begin{align*}
\|N(u)(t)\|_{L^q_s} 
&\le C \int_0^t (t-\tau)^{-\frac{d(\alpha-1)}{2q} - \frac12 \{(\alpha-1)s - \gamma\}} 
	\| |\cdot|^{\gamma} F(u(\tau)) \|_{L^{\frac{q}{\alpha}}_{\sigma}} d\tau \\
\end{align*}
by Lemma \ref{l:wLpLq} with $q\equiv q,$ $p\equiv \frac{q}{\alpha},$ 
$s\equiv s$ and $s'\equiv \sigma := \alpha s-\gamma,$ 
provided that $1\le \frac{q}{\alpha} \le q \le \infty$ and 
$-\frac{d}{q} < s \le \alpha s -\gamma < d(1-\frac{\alpha}{q}),$ i.e., 
\begin{equation}\label{l:Kato.est:pr1}
	\alpha \le q \le \infty, \quad 
	\frac{\gamma}{\alpha-1} \le s
		\quad\text{and}\quad
	-\frac{d}{q} < s < \frac{\gamma+d}{\alpha}-\frac{d}{q}. 
\end{equation}
As $\| |\cdot|^{\gamma} F(u) \|_{L^{\frac{q}{\alpha}}_{\sigma}}= \| u \|_{L^q_s}^{\alpha}, $
we have 
\begin{align*}
\|N(u)(t)\|_{L^q_s} 
&\le C \int_0^t (t-\tau)^{-\frac{d(\alpha-1)}{2q} - \frac12 \{(\alpha-1)s - \gamma\}} 
	\tau^{-\frac{(s_c-s)\alpha}{2}} d\tau \times  \|u \|_{\mathcal{K}^s(T)}^{\alpha}, 
\end{align*}
where the last integral is bounded by 
\begin{equation*}
	t^{-\frac{s_c-s}2} B\left(\frac{\alpha-1}2 (s_c -s), 1- \frac{(s_c-s)\alpha}2\right),
\end{equation*}
where $B:(0,\infty)^2\rightarrow \R_{>0}$ is the beta function given by $B(x,y):=\int_0^1t^{x-1}(1-t)^{y-1}dt$, which is convergent if and only if 
\begin{equation}\label{l:Kato.est:pr2}
	s_c - \frac2{\alpha} < s < s_c.
\end{equation}
Gathering \eqref{l:Kato.est:pr1} and \eqref{l:Kato.est:pr2}, 
we have condition \eqref{l:Kato.est.c2}. For such an $s$ to exist, 
it suffices to take $\gamma,$ $\alpha$ and $q$ so that 
conditions \eqref{t:HH.LWP.c0}  and \eqref{l:Kato.est.c1} are met. 

	\smallbreak
We next show \eqref{l:Kato.est2}. 
Since there exists a constant $C=C(\alpha)$ such that 
\begin{equation}\label{diff.pt.est}
	|F(u)-F(v)| \le C (|u|^{\alpha-1} + |v|^{\alpha-1})|u-v| 
	\quad\text{for all} \quad u,v \in \mathbb{C},
\end{equation}
we have 
\begin{align*}
\|N(u)(t) - N(v)(t)\|_{L^q_s} 
&\le C \int_0^t (t-\tau)^{-\frac{d(\alpha-1)}{2q} - \frac12 \{(\alpha-1) (\theta s + (1-\theta) s_c) -\gamma\}} \\
	& \quad \times \left\| |\cdot|^{\gamma} 
	(|u|^{\alpha-1} + |v|^{\alpha-1})|u-v| \right\|_{L^{\frac{q}{\alpha}}_{\sigma}} 
	d\tau,
\end{align*}
thanks to Lemma \ref{l:wLpLq} with $q\equiv q,$ $p\equiv \frac{q}{\alpha},$ 
$s\equiv s$ and $s'\equiv \sigma :=  (\alpha-1) (\theta s + (1-\theta) s_c) +s-\gamma,$ 
provided that $1\le \frac{q}{\alpha} \le q \le \infty$ and 
$-\frac{d}{q} < s \le (\alpha-1) (\theta s + (1-\theta) s_c) +s-\gamma < d(1-\frac{\alpha}{q}),$ $\theta \in (0,1],$ i.e., 
\begin{equation}\label{l:Kato.est:pr1'}
\begin{aligned}
&	\alpha \le q \le \infty, \quad 
	s_c - \frac{d}{\theta} \left( \frac{2}{d(\alpha-1)} -\frac{1}{q} \right) \le s\\
		\text{and}\quad
&	-\frac{d}{q} < s < s_c + \frac1{1+\theta(\alpha-1)}
	\left( d-2 -\frac{2+\gamma}{\alpha-1} \right). 
\end{aligned}
\end{equation}
By H\"older's inequality with 
$\frac{\alpha}{q} = \frac{\theta(\alpha-1)}{q} + \frac{(1-\theta)(\alpha-1)}{q} + \frac1{q},$  
we have 
\begin{align*}
& \left\| |\cdot|^{\gamma} 
	(|u|^{\alpha-1} + |v|^{\alpha-1})|u-v| \right\|_{L^{\frac{q}{\alpha}}_{\sigma}}  \\
&\le \left( \|u\|_{L^q_s} + \|v\|_{L^q_s} \right)^{\theta(\alpha-1)}
	\left( \|u\|_{L^q_{s_c}} + \|v\|_{L^q_{s_c}} \right)^{(1-\theta)(\alpha-1)} 
		\, \|u-v\|_{L^q_s}.
\end{align*}
Thus, 
\begin{align*}
&\|N(u)(t) - N(v)(t)\|_{L^q_s} \\
&\le C t^{-\frac{s_c-s}2}
	B\left(\theta\frac{\alpha-1}2 (s_c -s), 1- \frac{\theta (\alpha -1)+ 1}2 (s_c-s)\right) \\
&\times \left(\|u\|_{\mathcal{K}^s(T)}
				+\|v\|_{\mathcal{K}^s(T)} \right)^{\theta(\alpha-1)} 
		\left(\|u\|_{L^\infty(0,T; L^q_{s_c}) }
				+\|v\|_{L^\infty(0,T; L^q_{s_c}) } \right)^{(1-\theta)(\alpha-1)} 
 	\|u-v\|_{\mathcal{K}^s(T)} \\
\end{align*}
in which the last beta function is convergent if $\theta>0$ and 
\begin{equation}\label{l:Kato.est:pr2'}
	s_c - \frac2{\theta(\alpha-1)+1} < s < s_c.
\end{equation}
Gathering \eqref{l:Kato.est:pr1'} and \eqref{l:Kato.est:pr2'}, we deduce that 
the restrictions for $s$ are \eqref{l:Kato.est.c2'}. 
Consequently, for such an $s$ to exist, it suffices to take $q$ such that \eqref{l:Kato.est.c1'}. 
Finally, for such a $q$ to exist, one must have 
$0<\frac1{d(1+\theta(\alpha-1))} \{ \frac2{\alpha-1} 
	+ \theta ( d- \frac{2+\gamma}{\alpha-1} ) \},$ i.e., 
$\alpha>1+\frac{2+\gamma}{d} - \frac2{\theta d}$ and $0 < \frac2{d(\alpha-1)},$ 
both of which hold thanks to \eqref{t:HH.LWP.c0}. This concludes the proof of the lemma. 
\end{proof}

The following is the stability estimate for the critical norm. 
\begin{lem}	\label{l:crt.est}
Let $T \in (0,\infty]$ and $d\in\mathbb{N}.$ 
Let $\gamma\in\R$ and $\alpha\in\R$ satisfy \eqref{t:HH.LWP.c0} . 
Let $q\in [1,\infty]$ be such that 
\begin{equation}\label{l:crt.est.c1}
	\alpha\le q \le \infty 
		\quad\text{and}\quad
	\frac1{q} < \min \left\{ \frac{2}{d(\alpha-1)}, \,
		\frac{2}{d(\alpha-1)} + \frac{(d-2)\alpha - d -\gamma}{d(\alpha-1)^2} \right\}
\end{equation}
and let $s \in \R$ be such that 
\begin{equation}\label{l:crt.est.c2}
	 s_c - \frac{d(\alpha-1)}{\alpha} \left(\frac{2}{d(\alpha-1)} - \frac1{q} \right) \le s 
	< \min \left\{ s_c, \, s_c + \frac{(d-2)\alpha - d -\gamma}{\alpha(\alpha-1)} \right\}, 
\end{equation}
where $s_c$ is as in \eqref{d:sc}. 
Then there exists a positive constant $C_2$ 
depending only on $d,$ $\alpha,$ $\gamma,$ $q$ and $s$ such that 
the map $N$ defined by \eqref{mapN} satisfies 
\begin{equation}\nonumber
	\|N(u)\|_{L^\infty(0,T ; L^q_{s_c})} 
		\le C_2 \|u\|_{\mathcal{K}^s(T)}^{\alpha}
\end{equation}
for all $u,v \in \mathcal{K}^s(T).$ 
\end{lem}
\begin{proof}
Let $T>0$ and $u,v \in \mathcal{K}^s(T).$ We have 
\begin{align*}
\|N(u)(t)\|_{L^q_{s_c}} 
&\le C \int_0^t (t-\tau)^{-\frac{d(\alpha-1)}{2q} - \frac12 (\alpha s -\gamma- s_c)} 
	 \| u(\tau) \|_{L^q_s}^{\alpha} d\tau \\
&\le C B\left(\frac{\alpha}2 (s_c-s), 1 - \frac{(s_c-s)\alpha}2 \right)  
	\times  \|u \|_{\mathcal{K}^s(T)}^{\alpha}, 
\end{align*}
thanks to Lemma \ref{l:wLpLq} with $q\equiv q,$ $p\equiv \frac{q}{\alpha},$ 
$s\equiv s_c$ and $s'\equiv \alpha s-\gamma,$ 
provided that $1\le \frac{q}{\alpha} \le q \le \infty$ and 
$-\frac{d}{q} < s_c \le \alpha s -\gamma < d(1-\frac{\alpha}{q}),$ i.e., 
\begin{equation}\nonumber
	-2<\gamma, \quad 
	\alpha \le q \le \infty	\quad\text{and}\quad 
	\frac{s_c+\gamma}{\alpha} \le s < \frac{d+\gamma}{\alpha}-\frac{d}{q}. 
\end{equation}
The final beta function is convergent if \eqref{l:Kato.est:pr2} holds. 
Since $s_c - \frac2{\alpha} \le \frac{s_c+\gamma}{\alpha},$
the restrictions on $s$ are \eqref{l:crt.est.c2}. 
For such an $s$ to exist, $q$ must satisfy \eqref{l:crt.est.c1} in addition to 
$\alpha\le q \le\infty.$ 
Indeed, $\frac{s_c + \gamma}{\alpha}<s_c$ 
is equivalent to $\frac1{q} < \frac2{d(\alpha-1)}$ and 
$\frac{s_c+\gamma}{\alpha} < s_c$ is equivalent to 
$\frac1{q} < \frac1{\alpha-1} \left(1-\frac{2+\gamma}{d(\alpha-1)} \right).$ 
This completes the proof of the lemma. 
\end{proof}
\begin{rem}	\label{r:crt.est2}
Note that the above lemma along with Lemma \ref{l:wLpLq} imply that 
a solution $u \in \mathcal{K}^s(T)$ yields the regularity 
$u\in C([0,T] ; L^q_{s_c}(\R^d)),$ if $u_0 \in L^{q}_{s_c}(\R^d).$ Thus, 
if we allow the abuse of notation, the equivalence 
$\mathcal{K}^s(T) = C([0,T] ; L^q_{s_c}(\R^d)) \cap \mathcal{K}^s(T)$ 
holds as solution spaces of \eqref{HH}. 
\end{rem}

\subsubsection{Subcritical case}
The following are the stability and contraction estimates in the subcritical regime. 
\begin{lem}
\label{l:Kato.est.sub}
Let $T \in (0,\infty]$ and $d\in\mathbb{N}.$ Let $\gamma\in\R$ and $\alpha\in\R$ satisfy 
\eqref{t:HH.LWP.c0}. 
Fix $\tilde s\in\R$ so that 
\begin{equation}\label{l:Kato.est.sub.c0}
	\tilde s < \frac{2+\gamma}{\alpha-1}.
\end{equation}
Let $q\in [1,\infty]$ be such that 
\begin{equation}\label{l:Kato.est.sub.c1}
	\alpha\le q \le \infty
		\quad\text{and}\quad
	\frac1{q} < \min \left\{ \frac{2}{d(\alpha-1)}, \,
		\frac1{\alpha} \left(1 - \frac{\gamma}{d(\alpha-1)}\right), \, 
		\frac1{d} \left(\frac{2+\gamma}{\alpha-1} -\tilde s \right)\right\}
\end{equation}
and let $s \in \R$ be such that 
\begin{equation}\label{l:Kato.est.sub.c2}
	\frac{\gamma}{\alpha-1}\le s
		\quad\text{and}\quad
	\max\left\{\tilde s - \frac2{\alpha}, \, - \frac{d}{q} \right\} < s 
	< \min \left\{ \frac{d+\gamma}{\alpha} - \frac{d}{q}, \, s_c \right\}, 
\end{equation}
where $s_c$ is as in \eqref{d:sc}, 
Then there exist positive constants $\tilde C_0$ and $\tilde C_1$ 
depending only on $d,$ $\alpha,$ $\gamma,$ $q,$ $\tilde s$ and $s$ such that 
the map $N$ defined by \eqref{mapN} satisfies 
\begin{equation}\label{l:Kato.est.sub1}
\|N(u)\|_{\tilde{\mathcal{K}}^s(T)} 
	\le \tilde C_0 T^{\frac{\alpha-1}2(s_c-\tilde s)} \|u\|_{\tilde{\mathcal{K}}^s(T)}^{\alpha}
\end{equation}
and 
\begin{equation}\label{l:Kato.est.sub2}
\|N(u) - N(v)\|_{\tilde{\mathcal{K}}^s(T)} 
	\le \tilde C_1 T^{\frac{\alpha-1}2(s_c-\tilde s)} \left( \|u\|_{\tilde{\mathcal{K}}^s(T)}^{\alpha-1} + \|v\|_{\tilde{\mathcal{K}}^s(T)}^{\alpha-1} \right)
	\|u-v\|_{\tilde{\mathcal{K}}^s(T)}
\end{equation}
for all $u,v \in \tilde{\mathcal{K}}^s(T).$ 
\end{lem}
\begin{rem}
Note that $\frac1{q}<\frac1{d} \left(\frac{2+\gamma}{\alpha-1} -\tilde s \right)$ 
in \eqref{l:Kato.est.sub.c1} amounts to $\tilde s<s_c,$ 
so the power of $T$ in \eqref{l:Kato.est.sub1} and \eqref{l:Kato.est.sub2} is positive. 
\end{rem}
\begin{proof}[Proof of Lemma \ref{l:Kato.est.sub}]
We have 
\begin{align*}
\|N(u)(t)&\|_{L^q_{s}} 
\le C \int_0^t (t-\tau)^{-\frac{d(\alpha-1)}{2q} - \frac12 ((\alpha-1) s -\gamma)} 
	 \| u(\tau) \|_{L^q_s}^{\alpha} d\tau \\
&\le C t^{-\frac12(s-\tilde s)} t^{\frac{\alpha-1}2(s_c-\tilde s)} B\left(\frac{(\alpha-1)}2 (s_c-s), 1 - \frac{(\tilde s-s)\alpha}2 \right)  
	\times  \|u \|_{\tilde{\mathcal{K}}^s(T)}^{\alpha}, 
\end{align*}
thanks to Lemma \ref{l:wLpLq} with
$q\equiv q,$ $p\equiv \frac{q}{\alpha},$ 
$s\equiv \tilde s$ and $s'\equiv \alpha s-\gamma,$ 
provided that $1\le \frac{q}{\alpha} \le q \le \infty$ and 
$-\frac{d}{q} < s \le \alpha s -\gamma < d(1-\frac{\alpha}{q}),$ i.e., 
\begin{equation}\nonumber
	\alpha \le q \le \infty, \quad
	\frac{\gamma}{\alpha-1} \le s	\quad\text{and}\quad 
	-\frac{d}{q} < s < \frac{d+\gamma}{\alpha}-\frac{d}{q}. 
\end{equation}
The final beta function is convergent if $\tilde s-\frac2{\alpha}<s< s_c.$
Thus, the restrictions on $s$ are \eqref{l:Kato.est.sub.c2}. 
For such an $s$ to exist, $q$ must satisfy \eqref{l:Kato.est.sub.c1}. 
Finally, for such a $q$ to exist, 
we immediately see that $\tilde s$ must satisfy \eqref{l:Kato.est.sub.c0}. 

The proof for the difference is similar to the above so we omit the details. 
This completes the proof of the lemma. 
\end{proof}

\begin{lem}
\label{l:subcrt.est}
Let $T \in (0,\infty]$ and $d\in\mathbb{N}.$ 
Let $\gamma\in\R$ and $\alpha\in\R$ satisfy \eqref{t:HH.LWP.c0}. 
Fix $\tilde s$ so that \eqref{t:HH.LWP.sub.cs} is satisfied. 
Let $q\in [1,\infty]$ be such that 
\begin{equation}\label{l:subcrt.est.c1}
	\alpha\le q \le \infty
		\quad\text{and}\quad
	-\frac{\tilde s}{d} <\frac1{q} 
	<\min\left\{ \frac1{\alpha} \left(1 -\frac{\tilde s}{d}\right), \, 
			\frac1{d} \left(\frac{2+\gamma}{\alpha-1} -\tilde s \right) \right\}
\end{equation}
and let $s \in \R$ be such that 
\begin{equation}\label{l:subcrt.est.c2}
	\frac{\tilde s+\gamma}{\alpha} \le s 
	< \min \left\{ \frac{d+\gamma}{\alpha} - \frac{d}{q}, \tilde s \right\}, 
\end{equation}
where $s_c$ is as in \eqref{d:sc}. 
Then there exists a positive constant $\tilde C_2$ 
depending only on $d,$ $\gamma,$ $\alpha,$ $\tilde s,$ $q$ and $s$ such that 
the map $N$ defined by \eqref{mapN} satisfies 
\begin{equation}\nonumber 
	\|N(u)\|_{L^\infty(0,T ; L^q_{\tilde s})} 
		\le \tilde C_2 T^{\frac{\alpha-1}2(s_c-\tilde s)} \|u\|_{\tilde{\mathcal{K}}^s(T)}^{\alpha}
\end{equation}
for all $u\in \tilde{\mathcal{K}}^s(T).$ 
\end{lem}
\begin{proof}
We have 
\begin{align*}
\|N(u)(t)&\|_{L^q_{\tilde s}} 
\le C \int_0^t (t-\tau)^{-\frac{d(\alpha-1)}{2q} - \frac12 (\alpha s -\gamma- \tilde s)} 
	 \| u(\tau) \|_{L^q_s}^{\alpha} d\tau \\
&\le C t^{\frac{\alpha-1}2(s_c-\tilde s)} B\left(\frac12 \{(\alpha-1)(s_c-s) + \tilde s-s\}, 1 - \frac{(\tilde s-s)\alpha}2 \right)  
	\times  \|u \|_{\tilde{\mathcal{K}}^s(T)}^{\alpha}, 
\end{align*}
thanks to Lemma \ref{l:wLpLq} with
$q\equiv q,$ $p\equiv \frac{q}{\alpha},$ 
$s\equiv \tilde s$ and $s'\equiv \alpha s-\gamma,$ 
provided that $1\le \frac{q}{\alpha} \le q \le \infty$ and 
$-\frac{d}{q} < \tilde s \le \alpha s -\gamma < d(1-\frac{\alpha}{q}),$ i.e., 
\begin{equation}\nonumber
	\alpha \le q \le \infty, \quad
	-\frac{\tilde s}{d} <\frac1{q}		\quad\text{and}\quad 
	\frac{\tilde s+\gamma}{\alpha} \le s < \frac{d+\gamma}{\alpha}-\frac{d}{q}. 
\end{equation}
The final beta function is convergent if $\tilde s-\frac2{\alpha}<s<\tilde s< s_c.$
Since $\tilde s-\frac2{\alpha}<\frac{\tilde s+\gamma}{\alpha}$ 
(by $\tilde s < s_c < \frac{2+\gamma}{\alpha-1}$), the restrictions on $s$ are 
\eqref{l:subcrt.est.c2}. 
For such an $s$ to exist, $q$ must satisfy \eqref{l:subcrt.est.c1} 
and $\frac{\gamma}{\alpha-1} < \tilde s.$ 
Finally, for such a $q$ to exist, $\tilde s$ must satisfy \eqref{t:HH.LWP.sub.cs}
since 
\[
	-\frac{d}{\alpha-1}
	<\max\left\{ -\frac{d}{\alpha}, \, \frac{\gamma}{\alpha-1} \right\} \quad\text{and}\quad
	\frac{2+\gamma}{\alpha-1} < d.
\]
This completes the proof of the lemma. 
\end{proof}

\subsubsection{Upgrade of regularity}
The following lemma is used to show the regularity of the $L^q_{\tilde{s}}(\R^d)$-mild solution. 
\begin{lem}	\label{l:b.strap}
Let $p,q \in [1,\infty]$ and $s,s' \in \R.$ Under condition \eqref{t:HH.LWP.c0}, 
let pairs $(q,s)$ and $(p,s')$ be such that either
\begin{equation}	\label{l:b.strap:c}
\begin{aligned}
	&\alpha \le q < \infty, \quad
	\max\left\{- \frac{d}{q}, \, \frac{1}{\alpha}\left( \gamma- \frac{d}{q} \right)\right\} < s 
	<\frac{d+\gamma}{\alpha} - \frac{d}{q}, \\
	&\max\left\{0, \, -\frac{s}{d}, \, \frac{\gamma-\alpha s}{d} \right\}
		<\frac1{p} \le \frac{1}{q}, \quad
	 -\frac{d}{p} < s' \le \min\{s, \, \alpha s -\gamma\}. 
\end{aligned}
\end{equation}
or 
\begin{equation}	\label{l:b.strap:p=inf}
\begin{aligned}
	&\alpha < q \le \infty, \quad
	\max\left\{0,\, \frac{\gamma}{\alpha}\right\} \le s 
	<\frac{d+\gamma}{\alpha} - \frac{d}{q},\\
	&p= \infty, \quad 0\le s' \le \min\{s, \, \alpha s -\gamma\}.
\end{aligned}
\end{equation}
Let $u$ be the $L^q_{s_c}(\R^d)$-mild solution of \eqref{HH} with initial data $u_0 \in \mathcal{S}'(\R^d)$ on $[0,T_m)$
such that 
\begin{equation}\nonumber
	\sup_{t\in [0,T_m)} t^{\frac{s_c(q)-s}{2} } \|u(t)\|_{L^q_{s}} < \infty.
\end{equation}
Then it follows that 
\begin{equation}\nonumber
	\sup_{t\in [0,T_m)} t^{\frac{s_c(p)-s'}{2} } \|u(t)\|_{L^p_{s'}} < \infty.
\end{equation}
\end{lem}
\begin{proof}
We use a similar argument as in \cite{SnoTayWei2001} (See also \cite{BenTayWei2017}). 
Let 
\[
	A := \sup_{t\in [0,T_m)} t^{\frac{s_c(q)-s}2}\|u(t)\|_{L^{q}_{s}}<\infty.
\]
Let $t\in (0,T_m)$. We use the integral representation 
\begin{equation}\nonumber
	u(t) = e^{\frac{t}2 \Delta} u(t/2) 
	+ \int_{\frac{t}{2}}^t e^{(t-\tau)\Delta} \left\{ |\cdot|^{\gamma} 
		F(u(\tau,\cdot)) \right\} d\tau. 
\end{equation}
It follows from Lemma \ref{l:wLpLq} with $q\equiv p,$ $p\equiv q,$ 
$s'\equiv s'$ and $s\equiv s,$ that 
\begin{align*}
	\|e^{\frac{t}2\Delta} u(t/2)\|_{L^{p}_{s'}} 
	\le C t^{-\frac{d}{2}(\frac1{q}-\frac1{p}) - \frac{s-s'}2} \|u(t/2)\|_{L^{q}_{s}} 
	\le C t^{-\frac{1}{2}(\frac{2+\gamma}{\alpha-1}-\frac{d}{p} - s')} A 
\end{align*}
if 
\begin{equation}\label{l:b.strap:pr0}
\begin{aligned}
	&1\le q \le p \le \infty
		\quad\text{and}\quad 
	-\frac{d}{p} < s' \le s < d\left(1-\frac1{q}\right) \quad\left(0\le s' \text{ if } p=\infty\right). \\
\end{aligned}
\end{equation}
On the other hand, 
\begin{equation}\nonumber
\begin{aligned}
\|&N(u)(t)\|_{L^{p}_{s'}} 
\le  C \int_{\frac{t}{2}}^t 
		(t-\tau)^{-\frac{d}2(\frac{\alpha}{q}-\frac1{p}) - \frac{\alpha s -\gamma-s'}2} 
		 \| u(\tau) \|_{L^q_s}^{\alpha} d\tau \\
&\le CA^{\alpha} \int_{\frac{t}{2}}^t (t-\tau)^{-\frac{d}2(\frac{\alpha}{q}-\frac1{p}) - \frac{\alpha s -s'-\gamma}2} 
	\tau^{-\frac{(s_c-s)\alpha}{2}} d\tau \\
&= CA^{\alpha}  t^{-\frac12(\frac{2+\gamma}{\alpha-1} - \frac{d}{p} - s')} 
 \int_{\frac{1}{2}}^1 (1-\tau)^{-\frac{d}2(\frac{\alpha}{q}-\frac1{p}) - \frac{\alpha s -s'-\gamma}2} 
	\tau^{-\frac{(s_c-s)\alpha}{2}} d\tau, 
\end{aligned}
\end{equation}
thanks to Lemma \ref{l:wLpLq} with $q\equiv p,$ $p\equiv \frac{q}{\alpha},$ 
$s'\equiv s'$ and $s\equiv \alpha s-\gamma,$ provided that 
\begin{equation}\label{l:b.strap:pr1}
\begin{aligned}
	&1\le \frac{q}{\alpha} \le p \le \infty
		\quad\text{and}\quad 
	-\frac{d}{p} < s' \le \alpha s-\gamma < d\left(1-\frac{\alpha}{q}\right)
	\quad\left(0\le s' \text{ if } p=\infty\right).\\
\end{aligned}
\end{equation}
i.e., $\alpha \le q \le \infty,$ $\frac1{p} \le \frac{\alpha}{q},$
$-\frac{d}{p} < s' \le \alpha s-\gamma$ ($0\le s'\le \alpha s-\gamma$ if $p=\infty$) and $s < \frac{d+\gamma}{\alpha} - \frac{d}{q}.$ 
The final integral is convergent if 
\begin{equation}\label{l:b.strap:pr2}
	1-\frac{d}2\left(\frac{\alpha}{q}-\frac1{p}\right) - \frac{\alpha s-s' -\gamma}2>0, 
		\quad\text{i.e.,}\quad
	\alpha\left( \frac{d}{q} + s\right) - 2 - \gamma - \frac{d}{p}<s'.
\end{equation}
Thus, we have 
\begin{equation}\nonumber
	\sup_{t\in [0,T_m)} t^{\frac{s_c(p)-s'}{2} } \|u(t)\|_{L^p_{s'}} 
	\le C (A + A^{\alpha})
\end{equation}
under \eqref{l:b.strap:pr0}, \eqref{l:b.strap:pr1} and \eqref{l:b.strap:pr2}. 
Since $s < \frac{d+\gamma}{\alpha} - \frac{d}{q}$ from \eqref{l:b.strap:pr1}, we have 
$\alpha( \frac{d}{q} + s) - 2 - \gamma - \frac{d}{p} < - \frac{d}{p}.$ Thus, 
the conditions for $s'$ are that in \eqref{l:b.strap:c}. 
By tedious but straightforward computations, 
we may easily see that under condition \eqref{t:HH.LWP.c0}, 
the necessary and sufficient conditions of 
\eqref{l:b.strap:pr0}, \eqref{l:b.strap:pr1} and \eqref{l:b.strap:pr2} are 
\eqref{l:b.strap:c} or \eqref{l:b.strap:p=inf}. Hence, the lemma is proved. 
\end{proof}

\section{Local well-posedness and self-similar solutions}

\subsection{Proof of Theorem \ref{t:HH.LWP}}
In order to prove Theorem \ref{t:HH.LWP}, we prepare the following lemma. 
\begin{lem}\label{l:exist.crt}
Let positive numbers $\rho>0$ and $M>0$ satisfy 
\begin{equation}\label{l:exist.crt.c0}
	\rho + C_0 M^\alpha \le M \quad\text{and}\quad
	2 C_1 M^{\alpha-1} <1,
\end{equation}
where $C_0$ and $C_1$ are as in Lemma \ref{l:Kato.est}. 
Under conditions \eqref{t:HH.LWP.c0}, \eqref{t:HH.LWP.c1} and \eqref{t:HH.LWP.c2}, 
let $T\in (0,\infty]$ and $u_0 \in \mathcal{S}'(\R^d)$ be 
such that $e^{t\Delta}u_0 \in \mathcal{K}^s(T).$ 
If $\|e^{t\Delta}u_0\|_{\mathcal{K}^s(T)}\le \rho,$ 
then a solution $u$ to \eqref{HH} exists 
such that $u -e^{t\Delta} u_0 \in 
L^\infty(0,T ; L^q_{s_c}(\R^d)) \cap C((0,T] ; L^q_{s_c}(\R^d))$ and 
$\|u\|_{\mathcal{K}^s(T)} \le M.$ 
Moreover, the solution satisfies the following properties:
\begin{enumerate}[$(i)$]
\item $u -e^{t\Delta} u_0 \in L^\infty(0,T ; L^q_{\sigma}(\R^d))$ for $\sigma$ such that 
\begin{equation}	\label{l:exist.crt.csig}
	s_c \le \sigma \le \alpha s -\gamma.
\end{equation}
\item 
$u -e^{t\Delta} u_0 \in C([0,T) ; L^q_{\sigma}(\R^d))$ 
and $\Lim_{t\to0} \|u(t) -e^{t\Delta} u_0\|_{L^q_{\sigma}} = 0$ for $\sigma$ such that 
\eqref{l:exist.crt.csig} and $\sigma>s_c.$ 
\item $\displaystyle \lim_{t\to0} u(t) = u_0$ in the sense of distributions. 
\item Let $\gamma\ge0.$ Then the solution $u$ satisfies 
\[
\sup_{0<t<T} 
	t^{\frac{s_c(p_{\theta} ) - s_{\theta}}{2}} \|u(t)\|_{L^{p_{\theta}}_{s_{\theta}}} < \infty
\]
where 
\begin{equation}\nonumber	
	p_{\theta} = \frac{q}{\theta}, \quad 
	\theta s \le s_{\theta} \le s \quad\text{and}\quad 0\le \theta \le 1.
\end{equation}
In particular, if $\gamma\ge0,$ $u(t) \in L^\infty(\R^d)$ for $t>0.$
\end{enumerate}
\end{lem}
\begin{rem}\label{r:exist}
To meet \eqref{l:exist.crt.c0}, it suffices to take $M=2\rho$ and 
\[
M < \min\left\{  (2C_0)^{-\frac{1}{\alpha-1}}, \, 
                        (2C_1)^{- \frac{1}{\alpha-1}} \right\}.
\]
\end{rem}
\begin{rem}
If $u_0\in L^q_{s_c}(\R^d)$, then $u_0$ satisfies the assumptions of 
Lemma \ref{l:exist.crt} with $T=\infty.$ Indeed, letting $q\equiv q,$ $p\equiv q,$ 
$s\equiv s_c$ and $s'\equiv s$ in Lemma \ref{l:wLpLq}, we obtain 
\begin{equation}\nonumber
	\|e^{t\Delta} u_0\|_{L^q_s} \le C t^{-\frac{s_c-s}2} \|u_0\|_{L^q_{s_c}}
\end{equation}
provided that $-\frac{d}{q} < s\le s_c < d(1-\frac1{q}),$ i.e., 
$\gamma>-2$ and $\alpha>\alpha_F(d,\gamma).$ 
Thus, $e^{t\Delta} u_0 \in \mathcal{K}^s.$ 
\end{rem}
\begin{proof}[Proof of Lemma \ref{l:exist.crt}]
Setting the metric $d(u,v) := \|u-v\|_{\mathcal{K}^s(T)}$, we may show that $(\mathcal{K}^s(T),d)$ is a nonempty complete metric space. Let
$X_M := \{ u \in\mathcal{K}^s(T) \,;\,  \|u\|_{\mathcal{K}^s(T)} \le M \}$ be the closed ball in $\mathcal{K}^s(T)$ centered at the origin with radius $M$.
We prove that the map defined in \eqref{map} has a fixed point in $X_M.$ 
Thanks to Lemma \ref{l:Kato.est} and \eqref{l:exist.crt.c0}, we have 
\begin{equation}\nonumber
\begin{aligned}
\|\Phi (u)\|_{\mathcal{K}^s(T)} 
\le \|e^{t\Delta} u_0 \|_{\mathcal{K}^s(T)} 
	+ C_0 \|u\|_{\mathcal{K}^s(T)}^\alpha 
\le \rho + C_0 M^\alpha \le M
\end{aligned}\end{equation} 
and 
\begin{equation}\label{t:HH.LWP.pr.Lip'}
\|\Phi (u)-\Phi (v)\|_{\mathcal{K}^s(T)} 
\le C_1 \left( \|u\|_{\mathcal{K}^s(T)}^{\alpha-1}
	+\|v\|_{\mathcal{K}^s(T)}^{\alpha-1} \right) 
	\|u-v\|_{\mathcal{K}^s(T)}
\le 2 C_1 M^{\alpha-1} \|u-v\|_{\mathcal{K}^s(T)}
\end{equation}
for any $u, v\in X_M,$ where $2 C_1 M^{\alpha-1}<1.$ 
These prove that $\Phi(u) \in X_M$ and that 
$\Phi$ is a contraction mapping in $X_M.$ 
Thus, Banach's fixed point theorem ensures the existence of 
a unique fixed point $u$ for the map $\Phi$ in $X_M,$ 
provided that $q$ and $s$ satisfy \eqref{l:Kato.est.c1} and \eqref{l:Kato.est.c2}. 
The fixed point $u$ also satisfies, by construction, the estimate 
$\|u\|_{\mathcal{K}^s(T)} \le M.$ 

Having obtained a fixed point in $\mathcal{K}^s(T)$ for some $T,$ 
we have $u -e^{t\Delta} u_0 \in L^\infty(0,T;L^q_{s_c}(\R^d))$ by Lemma 
\ref{l:crt.est}, provided further that \eqref{l:crt.est.c1} and \eqref{l:crt.est.c2} are satisfied. 
We see that $\frac1{q} < \frac{2}{d(\alpha-1)}$, 
$q>0,$ $\alpha>1$ and $\gamma>-2$ imply 
\[
	\max\left\{\frac{\gamma}{\alpha-1}, \, - \frac{d}{q} \right\} 
		< \frac{s_c+\gamma}{\alpha} = s_c - \frac{d(\alpha-1)}{\alpha} \left(\frac{2}{d(\alpha-1)} - \frac1{q} \right) 
\]
so $\frac{s_c+\gamma}{\alpha}$ is the stronger lower bound for $s.$ 
Thus, $s$ must satisfy \eqref{t:HH.LWP.c2}. 
Combining \eqref{l:Kato.est.c1} and \eqref{l:crt.est.c1}, we end up with 
\begin{equation}\label{t:HH.LWP:pr1}
	\frac1{q} < \min \left\{ \frac{2}{d(\alpha-1)}, \,
		\frac1{\alpha} \left(1 - \frac{\gamma}{d(\alpha-1)}\right), \, 
		\frac1{\alpha-1} \left(1-\frac{2+\gamma}{d(\alpha-1)} \right)\right\}, 
\end{equation}
which in fact amounts to \eqref{t:HH.LWP.c1}. 
	\smallbreak
We next prove the assertion $(i)$--$(iii).$ Fix a solution $u \in \mathcal{K}^s(T)$ 
with $q$ and $s$ as in \eqref{t:HH.LWP.c1} and \eqref{t:HH.LWP.c2}. We have 
\begin{equation}\label{t:HH.LWP:pr2}
\begin{aligned}
\|N(u)(t)\|_{L^q_{\sigma}} 
&\le C \int_0^t (t-\tau)^{-\frac{d(\alpha-1)}{2q} - \frac12 (\alpha s -\gamma- \sigma)} 
	 \| u(\tau) \|_{L^q_s}^{\alpha} d\tau \\
&\le C \int_0^t (t-\tau)^{-\frac{d(\alpha-1)}{2q} - \frac12 (\alpha s -\gamma- \sigma)} 
	\tau^{-\frac{(s_c-s)\alpha}{2}} d\tau \times  \|u \|_{\mathcal{K}^s(T)}^{\alpha} \\
&= C  t^{\frac{\sigma-s_c}2} B\left(\frac{(\alpha-1)(s_c-s) + \sigma-s}2, \, 1 - \frac{(s_c-s)\alpha}2 \right)  
	\times  \|u \|_{\mathcal{K}^s(T)}^{\alpha}, 
\end{aligned}
\end{equation}
thanks to Lemma \ref{l:wLpLq} with $q\equiv q,$ $p\equiv \frac{q}{\alpha},$ 
$s' \equiv \sigma$ and $s\equiv \alpha s-\gamma,$ 
provided that $1\le \frac{q}{\alpha} \le q \le \infty$ and 
$-\frac{d}{q} < \sigma \le \alpha s -\gamma < d(1-\frac{\alpha}{q}).$ 
The power of $t$ in the final line is non-negative if $\sigma \ge s_c.$ 
The use of Lemma \ref{l:wLpLq} along with the convergence of the beta function require, 
in addition to \eqref{t:HH.LWP.c1} and \eqref{t:HH.LWP.c2}, that $\sigma$ satisfies
\eqref{l:exist.crt.csig}. For such a $\sigma$ to exist, one needs 
$\frac{s_c+\gamma}{\alpha}\le s,$ which is assured by \eqref{t:HH.LWP.c2}. 
If $\sigma>s_c,$ (i.e., if $\frac{s_c+\gamma}{\alpha}< s,$) then the power of $t$ is positive, thus the right-hand side of 
\eqref{t:HH.LWP:pr2} goes to zero as $t\to0.$ Hence, the assertions $(ii)$ 
and $(iii)$ are proved. 
	\smallbreak
Finally, we prove the assertion $(iv).$ Fix a solution $u \in \mathcal{K}^s(T)$ with 
$q$ and $s$ as in \eqref{t:HH.LWP.c1} and \eqref{t:HH.LWP.c2}. 
Here, we notice that under $\gamma>0,$ the lower bound of \eqref{t:HH.LWP.c2} 
always satisfies 
\begin{equation}\nonumber
	\max\left\{0, \, \frac{\gamma}{\alpha}\right\} 
	\le s_c - \frac{d(\alpha-1)}{\alpha} \left(\frac{2}{d(\alpha-1)} - \frac1{q} \right),
\end{equation}
which implies that the condition 
\eqref{l:b.strap:p=inf} of Lemma \ref{l:b.strap} is always satisfied as well. Thus, 
Lemma \ref{l:b.strap} immediately implies 
\begin{equation}\nonumber
	\sup_{t\in [0,T)} t^{\frac{s_c(\infty)-s'}{2} } \|u(t)\|_{L^{\infty}_{s'}} < \infty
\end{equation}
for 
\begin{equation}\nonumber
	0\le s' \le \min\{s, \, \alpha s -\gamma\}.
\end{equation}
We also have $u\in \mathcal{K}^s(T)$ by assumption. Thus, the conclusion follows from 
Proposition \ref{p:wL.sp} $(3).$ 
\end{proof}
We start by proving the uniqueness of our solution. 
	\subsubsection{Proof of $(ii)$}
Let $T>0$ be given and fixed. We prove the uniqueness in $\mathcal{K}^s(T).$ Under conditions \eqref{t:HH.LWP.c0}, 
\eqref{t:HH.LWP.c1} and \eqref{t:HH.LWP.c2}, 
let $u$ and $v$ be two solutions to \eqref{integral-eq} 
belonging to $C([0,T] ; L^q_{s_c}(\R^d)) \cap \mathcal{K}^s(T)$ 
with the same initial data $u_0 \in L^q_{s_c}(\R^d)$ 
($u_0 \in \mathcal{L}^{\infty}_{s_c}(\R^d)$ if $q=\infty$)
\footnote{We assume $u_0 \in \mathcal{L}^{\infty}_{s_c}(\R^d)$ if $q=\infty$ in 
order to utilize the density, which is needed in the proof of \eqref{id.lmt.K} and \eqref{id.lmt.crt}.}
such that 
\begin{equation}\nonumber
	\|u\|_{\mathcal{K}^s(T)}+\|v\|_{\mathcal{K}^s(T)} \le K,
\end{equation}
for some positive constant $K.$ 
Let us recall that we have the following two limits at our disposal:
\begin{equation}\label{id.lmt.K}
	\lim_{T\to0} \|e^{t\Delta} u_0\|_{\mathcal{K}^s(T)} = 0
\end{equation}
and 
\begin{equation}\label{id.lmt.crt}
	\lim_{T\to0} \|u - e^{t\Delta} u_0\|_{L^\infty(0,T;L^q_{s_c})} = 0. 
\end{equation}
The former is the well-known fact stemming from the density of 
$C_0^\infty(\R^d)$ in $L^q_{s_c}(\R^d)$ (See Proposition \ref{p:wL.sp} in Appendix). 
The latter is shown by the triangle inequality and the continuity at $t=0$ of solutions 
for both the linear and nonlinear problems. 
Let $w :=u-v.$ By \eqref{diff.pt.est}, we have 
\begin{equation}\nonumber
	|F(u)-F(v)| 
	\le C |e^{t\Delta} u_0|^{\alpha-1} |u-v| 
	+ C (|u-e^{t\Delta} u_0|^{\alpha-1} + |v-e^{t\Delta} u_0|^{\alpha-1})|u-v|,
\end{equation}
which implies that $|w| \le C( I_1 + I_2 + I_3)$ (thanks to the maximum principle), where  
\begin{equation}\nonumber
\begin{aligned}
&	I_1 := \int_0^{t} e^{(t-\tau)\Delta} 
		\left\{|\cdot|^{\gamma} |e^{t\Delta} u_0|^{\alpha-1} |w| \right\} \, d\tau, \\
&	I_2 :=  \int_0^{t} e^{(t-\tau)\Delta} 
		\left\{|\cdot|^{\gamma} |u-e^{t\Delta} u_0|^{\alpha-1} |w| \right\} \, d\tau \\
		\text{and}\quad
&	I_3 :=  \int_0^{t} e^{(t-\tau)\Delta} 
		\left\{|\cdot|^{\gamma} |v-e^{t\Delta} u_0|^{\alpha-1} |w| \right\} \, d\tau. 
\end{aligned}
\end{equation}
Given $q$ and $s$ satisfying \eqref{t:HH.LWP.c1} and \eqref{t:HH.LWP.c2}, 
we may always choose $\theta$ so that \eqref{l:Kato.est.c1'} and \eqref{l:Kato.est.c2'} 
are satisfied. Indeed, \eqref{l:Kato.est.c1'} and \eqref{l:Kato.est.c2'} become 
\eqref{l:Kato.est.c1} and \eqref{l:Kato.est.c2} as $\theta\to1,$ respectively, 
which are weaker than the assumptions on $q$ and $s$ in Theorem \ref{t:HH.LWP}. 
The only condition that has to be considered independently is 
$s_c - \frac{d}{\theta} \left( \frac{2}{d(\alpha-1)} -\frac{1}{q} \right)  \le s$ in 
\eqref{l:Kato.est.c2'} (as this is not a strict inequality), 
but this causes no problem since 
$s_c - \frac{d}{\theta} \left( \frac{2}{d(\alpha-1)} -\frac{1}{q} \right) 
\le s_c - \frac{d(\alpha-1)}{\alpha} \left(\frac{2}{d(\alpha-1)} - \frac1{q} \right)$
holds for any $\theta \in (0,1].$ 
Thus, we may use estimate \eqref{l:Kato.est2} freely for our $q$ and $s.$ 

By the same calculation leading to \eqref{l:Kato.est1}, we deduce that 
\begin{equation}\label{pr.uni.1}
	\|I_1\|_{\mathcal{K}^s(T)} 
	\le C \|e^{t\Delta} u_0\|_{\mathcal{K}^s(T)}^{\alpha-1} \|w\|_{\mathcal{K}^s(T)}. 
\end{equation}
For $I_2,$ estimate \eqref{l:Kato.est2} implies 
\begin{equation}\label{pr.uni.2}
\begin{aligned}
\|I_2\|_{\mathcal{K}^s(T)} 
&\le C \|u-e^{t\Delta} u_0\|_{\mathcal{K}^s(T)}^{\theta(\alpha-1)} 
		\|u-e^{t\Delta} u_0\|_{L^\infty(0,T; L^q_{s_c}) }^{(1-\theta)(\alpha-1)} 
	 \|w\|_{\mathcal{K}^s(T)}  \\
&\le C K^{\theta(\alpha-1)} 
	\|u-e^{t\Delta} u_0\|_{L^\infty(0,T; L^q_{s_c}) }^{(1-\theta)(\alpha-1)} 
	\|w\|_{\mathcal{K}^s(T)}. 
\end{aligned}
\end{equation}
Similarly, we have 
\begin{equation}\label{pr.uni.3}
\begin{aligned}
\|I_3\|_{\mathcal{K}^s(T)} 
\le C K^{\theta(\alpha-1)} 
	\, \|v-e^{t\Delta} u_0\|_{L^\infty(0,T; L^q_{s_c}) }^{(1-\theta)(\alpha-1)} 
	\, \|w\|_{\mathcal{K}^s(T)}. 
\end{aligned}
\end{equation}
Gathering \eqref{pr.uni.1}, \eqref{pr.uni.2} and \eqref{pr.uni.3}, we deduce that 
there exists some positive constant $C$ 
independent of $T,$ $u_0,$ $u$ and $v$ such that 
\begin{equation}\nonumber
\|w\|_{\mathcal{K}^s(T)} 
\le C \mathcal{N}(T, u_0, u, v)
	\|w\|_{\mathcal{K}^s(T)} 
\end{equation}
where
\begin{equation}\nonumber
\mathcal{N}(T, u_0, u, v) := 
 \|e^{t\Delta} u_0\|_{\mathcal{K}^s(T)}^{\alpha-1} 
	+\|u-e^{t\Delta} u_0\|_{L^\infty(0,T; L^q_{s_c}) }^{(1-\theta)(\alpha-1)} 
	+ \|v-e^{t\Delta} u_0\|_{L^\infty(0,T; L^q_{s_c}) }^{(1-\theta)(\alpha-1)}.  
\end{equation}
Since $0<\theta<1,$ the above quantity goes to zero as 
$T$ tends to zero, thanks to \eqref{id.lmt.crt} and \eqref{id.lmt.K}. 
Thus, there exists some $T'$ such that 
\begin{equation}\nonumber
\|w\|_{\mathcal{K}^s(T')} 
	\le \frac1{2} \|w\|_{\mathcal{K}^s(T')}
\end{equation}
for instance, which implies the uniqueness on the interval $[0,T'].$ Set 
\begin{equation}\nonumber
T^* = \sup \{t\in [0,T] \,; \, u(\tau) = v(\tau) , \ 0\le \tau \le t\}.
\end{equation}
The preceding argument shows that $T^*>0.$
Now assume by contradiction that $T^*<T.$ By continuity of $u$ and $v,$ 
we have $u(T^*) = v(T^*).$ 
Setting $u^*(t) = u(t+T^*)$ and $v^*(t) = v(t+T^*),$ we may express the solutions as 
\begin{align*}\nonumber
&	u^*(t) = e^{t\Delta} u(T^*) + \int_0^t e^{(t-\tau)\Delta} 
	\left\{ |\cdot|^{\gamma} F(u(T^*+\tau,\cdot)) \right\} d\tau \\
		\text{and} \quad
&	v^*(t) = e^{t\Delta} u(T^*) + \int_0^t e^{(t-\tau)\Delta} 
	\left\{ |\cdot|^{\gamma} F(v(T^*+\tau,\cdot)) \right\} d\tau,  
\end{align*}
where $0\le t < T - T^*.$ 
By a similar calculation as above, we may show that 
\begin{equation}\nonumber
\|u^* - v^*\|_{\mathcal{K}^s(T)} 
	\le \mathcal{N}(T, u(T^*), u, v)  
		\|u^* - v^*\|_{\mathcal{K}^s(T)}, 
\end{equation}
which implies again that there exist some $T'$ such that $u^*(t) = v^*(t)$ 
for $t\in [0,T'],$ i.e., $u(t) = v(t)$ for $t \in (T^*, T^* +T'),$ a contradiction. 
Thus, $u(t)=v(t)$ on the whole interval $[0,T].$ 
This completes the proof of Theorem \ref{t:HH.LWP} $(ii).$  

	\subsubsection{Proof of $(i)$} 
Let $u_0 \in L^q_{s_c}(\R^d)$ 
($u_0 \in \mathcal{L}^{\infty}_{s_c}(\R^d)$ if $q=\infty$). 
We recall that $C_0^\infty(\R^d)$ is dense in the space $L^q_{s_c}(\R^d)$ 
by Proposition \ref{p:wL.sp}, which ensures the property \eqref{id.lmt.K}. 
Thus, there exists some real number $T$ that is 
small enough so that $\|e^{t\Delta} u_0\|_{\mathcal{K}^s(T)}\le \rho.$ 
Now Lemma \ref{l:exist.crt} asserts that 
\begin{equation}\nonumber
\|u\|_{L^\infty(0,T \,;\, L^q_{s_c})} 
	\le \|e^{t\Delta} u_0\|_{L^\infty(0,T \,;\, L^q_{s_c})} + C_2 \|u \|_{\mathcal{K}^s(T)}^{\alpha}
	\le \|u_0\|_{L^q_{s_c}} + C_2 M^{\alpha}. 
\end{equation}
The time-continuity at $t=0$ follows from a well-known argument 
(see \cites{OkaTsu2016, Tsu2011} for example). 
Thus, $u$ is an $L^q_{s_c}(\R^d)$-mild solution to \eqref{HH} on $[0,T]$ 
such that $\|u\|_{\mathcal{K}^s(T)}\le M.$ 
To deduce the estimate \eqref{t:HH.LWP.est}, 
it suffices to take $\rho=\|e^{t\Delta} u_0\|_{\mathcal{K}^s(T)}$ 
and $M$ as in Remark \ref{r:exist}. 
Given $u_0 \in L^q_{s_c}(\R^d),$ let the maximal existence time 
$T_m = T_m (u_0)$ be defined by \eqref{d:Tm} with $\tilde s = s_c.$ 
By a standard argument, uniqueness ensures that 
the solution can be extended to the maximal interval $[0,T_m).$ 

	\subsubsection{Proof of $(iii)$}
Given two initial data $u_0, v_0 \in L^q_{s_c}(\R^d),$ 
we next show the Lipschitz continuity of the flow map. 
Let $u$ and $v$ be two solutions associated with the initial data $u_0$ and $v_0,$ 
respectively, constructed in (i) with the estimate 
$\|u\|_{\mathcal{K}^s(T)}\le 2\|e^{t\Delta} u_0\|_{\mathcal{K}^s(T)}.$ 
Let $w := u-v$ and $w_0 := u_0-v_0.$ 
We carry out the same calculations as before to 
see that there exists a positive constant $C_3$ such that 
\begin{align*}
\|w\|_{L^\infty(0,T; L^q_{s_c})\cap \mathcal{K}^s(T)} 
&\le \|e^{t\Delta}w_0\|_{L^\infty(0,T; L^q_{s_c})\cap \mathcal{K}^s(T)} 
	+ C_3 \left( \|u \|_{\mathcal{K}^s(T)}^{\alpha-1}
		+\|v \|_{\mathcal{K}^s(T)}^{\alpha-1}\right) \|w\|_{\mathcal{K}^s(T)} \\
&\le \|w_0\|_{L^q_{s_c}} 
	+ 2 C_3 M^{\alpha-1} \|w\|_{\mathcal{K}^s(T)}, 
\end{align*}
where $M= \max \{ \|e^{t\Delta} u_0 \|_{\mathcal{K}^s(T)}, 
	\, \|e^{t\Delta} v_0 \|_{\mathcal{K}^s(T)}\}.$ 
By taking $T$ smaller if necessary 
($2 C_3 M^{\alpha-1}\le \frac12$ for instance), we deduce 
the Lipschitz stability on the short time-interval $[0,T].$ 

	\subsubsection{Proof of $(iv)$}
We prove the blow-up criterion by a contradiction argument. 
Let $T_m<\infty$ and suppose that $\|u\|_{\mathcal{K}^s(T_m)} <\infty$ holds.  
Let $u$ be a maximal solution and let $t_0 \in (0,T_m),$ to be fixed later.  
We aim to prove 
there exists an $\ep>0$ such that 
\begin{equation}\label{buc-aim}
	\|e^{t\Delta} u(t_0)\|_{\mathcal{K}^s(T_m-t_0+\ep)} \le \rho,
\end{equation}
where $\rho > 0$ is the constant as in Lemma \ref{l:exist.crt}. 
Once \eqref{buc-aim} is proved, 
the solution $u$ can be smoothly extended to $T_m+\ep.$ 
Moreover, $u$ is unique in 
$C([0,T_m+\ep] ; L^q_{s_c}(\R^d)) \cap \mathcal{K}^s(T_m+\ep)$ by $(ii),$ 
which contradicts the definition of $T_m.$ 
Thus, $\|u\|_{\mathcal{K}^s(T_m)}=\infty$ if $T_m<\infty.$ 

Let us concentrate on proving \eqref{buc-aim}. We may express the maximal solution as follows: 
\begin{equation}\nonumber
	u(t+t_0) = e^{t\Delta} u(t_0) + \int_0^t e^{(t-\tau)\Delta} \left\{ |\cdot|^{\gamma} F(u(t_0+\tau)) \right\} d\tau, \quad 0\le t < T_m - t_0.
\end{equation}
Thus, we have 
\begin{align*}\nonumber
	\|e^{t\Delta}u(t_0)&\|_{\mathcal{K}^s(T_m-t_0)}  \\
	&\le \|u(\cdot+t_0)\|_{\mathcal{K}^s(T_m-t_0)} 
	+ \left\| \int_0^t e^{(t-\tau)\Delta} \left\{ |\cdot|^{\gamma} F(u(t_0+\tau)) \right\} d\tau \right\|_{\mathcal{K}^s(T_m-t_0)}. 
\end{align*}
For the first term, we have 
\begin{equation}\label{pr.iv.1}
\begin{aligned}
	\|u(\cdot+t_0)&\|_{\mathcal{K}^s(T_m-t_0)} 
	= \sup_{0\le t \le T_m - t_0} t^{\frac{s_c-s}2} \|u(t + t_0)\|_{L^q}
	= \sup_{t_0\le s\le T_m} (s-t_0)^{\frac{s_c-s}2} \|u(s)\|_{L^q} \\
	&\le \left(\frac{T_m-t_0}{t_0} \right)^{\frac{s_c-s}2} 
		\sup_{t_0 \le s\le T_m} s^{\frac{s_c-s}2} \|u(s)\|_{L^q} 
	\le \left(\frac{T_m-t_0}{t_0} \right)^{\frac{s_c-s}2} \|u\|_{\mathcal{K}^s(T_m)}.
\end{aligned}
\end{equation}
For the second term, Lemma \ref{l:Kato.est} yields 
\begin{equation}\label{pr.iv.2}
	\left\| \int_0^t e^{(t-\tau)\Delta} \left\{ |\cdot|^{\gamma} 
				F(u(t_0+\tau)) \right\} d\tau \right\|_{\mathcal{K}^s(T_m-t_0)}
	\le C_0 \|u(\cdot+t_0)\|_{\mathcal{K}^s(T_m-t_0)}^\alpha.
\end{equation}
Since the right-hand sides in \eqref{pr.iv.1} and \eqref{pr.iv.2} go to $0$ as $t_0 \to T_m$,
we may fix some $t_0$ 
close enough to $T_m$ so that 
\begin{equation}\nonumber
	\|e^{t\Delta} u(t_0)\|_{\mathcal{K}^s(T_m-t_0)} 
	\le 2^{-\frac{s_c-s}2}  \frac{\rho}{2}.
\end{equation}
Let $\varepsilon\in (0, T_m -t_0)$, to be fixed later. 
Then, we have 
\begin{equation}\label{pr.iv.4}
\begin{aligned}
\sup_{2\ep\le t\le T_m - t_0 +\ep} t^{\frac{s_c-s}2} \|e^{t\Delta} u(t_0)\|_{L^q_s} 
	&=  \sup_{\ep \le s\le T_m - t_0} \left( \frac{s+\ep}{s}\right)^{\frac{s_c-s}2} 
		 s^{\frac{s_c-s}2} \|e^{(s+\ep)\Delta} u(t_0)\|_{L^q_s} \\
	&\le \sup_{\ep \le s\le T_m - t_0} \left( \frac{s+\ep}{s}\right)^{\frac{s_c-s}2} 
		\|e^{t\Delta} u(t_0)\|_{\mathcal{K}^s(T_m-t_0)}\\
	&\le 2^{\frac{s_c-s}2} 
		\|e^{t\Delta} u(t_0)\|_{\mathcal{K}^s(T_m-t_0)} \le  \frac{\rho}{2},
 \end{aligned}
 \end{equation}
where we have used 
$$\displaystyle{\sup_{\ep \le s\le T_m - t_0} \frac{s+\ep}{s} \le 2.}$$
On the other hand, since $u(t_0) \in L^q_{s_c}(\R^d),$ we may fix some $\ep>0$ 
such that 
\begin{equation}\label{pr.iv.0}
	\|e^{t\Delta}u(t_0)\|_{\mathcal{K}^s(2\ep)} \le \frac{\rho}2, 
\end{equation}
By \eqref{pr.iv.0} and \eqref{pr.iv.4}, we deduce that 
\[
\begin{split}
	\|e^{t\Delta} u(t_0)\|_{\mathcal{K}^s(T_m-t_0+\ep)} 
	& \le \|e^{t\Delta}u(t_0)\|_{\mathcal{K}^s(2\ep)} + \sup_{2\ep\le t\le T_m - t_0 +\ep} t^{\frac{s_c-s}2} \|e^{t\Delta} u(t_0)\|_{L^q_s} \\
	& \le \frac{\rho}2 + \frac{\rho}2 = \rho,
\end{split}
\]
which proves \eqref{buc-aim}. 

	\subsubsection{Proof of $(v)$}
Taking $T=\infty$ in Lemma \ref{l:exist.crt}, we deduce the global existence. 
Lastly, we show that if $T_m=\infty,$ then the solution is dissipative. 
We sketch the proof, as most of the computations are similar to the previous ones. 
We take $\{u_{0n}\}_{n\ge0} \subset C_0^\infty(\R^d)$ such that 
$u_{0n} \to u_0$ in $L^q_{s_c}(\R^d)$ and decompose the integral equation into 
\begin{equation}\nonumber\begin{aligned}
u(t) =  e^{t\Delta} u_{0n} + e^{t\Delta} (u_0-u_{0n}) 
&+ e^{(t-t')\Delta} \int_{0}^{t'} e^{(t'-\tau)\Delta} \left( |\cdot|^{\gamma} F(u(\tau)) \right)d\tau \\
&+ \int_{t'}^{t} e^{(t-\tau)\Delta} \left( |\cdot|^{\gamma} F(u(\tau)) \right)d\tau,
\end{aligned}\end{equation}
where $0 < t' < t.$ 
The first and second linear terms obviously tend to 0 as $n\to \infty$ and $t\to\infty.$
On the other hand, we may let $t'$ so close to $t$ so that the fourth term is small. 
Now that $t'$ is fixed, the third term can be written as $e^{(t-t')\Delta} f(t')$ with 
$f(t') \in L^q_{s_c}(\R^d)$, 
so we may use the semigroup property of $e^{t\Delta}$ and an 
approximation argument again. 
This completes the proof of the theorem.

\subsection{Proof of Theorem \ref{t:HH.LWP.sub}}
\begin{lem}\label{l:exist.sub}
Let real numbers $T\in (0,\infty),$ $\rho>0$ and $M>0$ satisfy 
\begin{equation}\label{l:exist.sub.c0}
	\rho + \tilde C_0 T^{\frac{\alpha-1}2(s_c-\tilde s)} M^\alpha \le M 
		\quad\text{and}\quad
	2 \tilde C_1 T^{\frac{\alpha-1}2(s_c-\tilde s)} M^{\alpha-1} <1,
\end{equation}
where $\tilde C_0$ and $\tilde C_1$ are as in Lemma \ref{l:Kato.est.sub}. 
Under conditions \eqref{t:HH.LWP.c0}, \eqref{t:HH.LWP.c1} and \eqref{t:HH.LWP.c2}, 
let $u_0 \in \mathcal{S}'(\R^d)$ be such that 
$e^{t\Delta}u_0 \in \tilde{\mathcal{K}}^s(T)$ for $T$ fixed as above.  
If $\|e^{t\Delta}u_0\|_{\tilde{\mathcal{K}}^s(T)}\le \rho,$ 
then a solution $u$ to \eqref{HH} exists 
such that $u -e^{t\Delta} u_0 \in C([0,T] ; L^q_{\tilde s}(\R^d))$ and 
$\|u\|_{\tilde{\mathcal{K}}^s(T)} \le M.$ 
\end{lem}
\begin{rem}\label{r:exist.sub}
To meet condition \eqref{l:exist.sub.c0}, it suffices to take $M=2\rho$ and 
$T$ such that 
\[
T < \min\left\{  (2^{\alpha}\tilde C_0)^{-\frac{2}{(\alpha-1)(s_c-\tilde s)}} , 
			\, (2^{\alpha}\tilde C_1)^{-\frac{2}{(\alpha-1)(s_c-\tilde s)}} 
			\right\} \rho^{-\frac2{s_c-\tilde s}}.
\]
\end{rem}
\begin{proof}[Proof of Lemma \ref{l:exist.sub}]
Setting the metric $d(u,v) := \|u-v\|_{\tilde{\mathcal{K}}^s(T)}$, we may show that $(\tilde{\mathcal{K}}^s(T),d)$ is a nonempty complete metric space. Let
$X_M := \{ u \in\mathcal{K}^s(T) \,;\,  \|u\|_{\tilde{\mathcal{K}}^s(T)} \le M \}$ be the closed ball in $\tilde{\mathcal{K}}^s(T)$ centered at the origin with radius $M$.
Similarly to the critical case, we may prove 
that the map defined in \eqref{map} has a fixed point in $\tilde X_M,$ 
thanks to Lemma \ref{l:Kato.est.sub} and \eqref{l:exist.sub.c0}. 
Thus, Banach's fixed point theorem ensures the existence of 
a unique fixed point $u$ for the map $\Phi$ in $\tilde X_M.$ 

Having obtained a fixed point in $\tilde{\mathcal{K}}^s(T),$ 
we deduce $u -e^{t\Delta} u_0 \in L^\infty(0,T;L^q_{\tilde s}(\R^d))$ 
thanks to Lemma \ref{l:subcrt.est}, provided further that 
\eqref{t:HH.LWP.sub.cs}, \eqref{l:subcrt.est.c1} and \eqref{l:subcrt.est.c2} 
are satisfied. 
We see that $s<\tilde s < s_c$ imply 
\[
	\max\left\{\frac{\gamma}{\alpha-1}, \, \tilde s - \frac2{\alpha} \right\} 
	< \frac{\tilde s+\gamma}{\alpha} 
\]
so $\frac{s_c+\gamma}{\alpha}$ is a new lower bound for $s.$
In conjunction with this stronger lower bound $\frac{s_c+\gamma}{\alpha} \le s,$ 
there also appears a new upper bound for $\frac1{q}.$ 
More precisely, for such an $s$ satisfying \eqref{l:Kato.est.sub.c2}
and \eqref{l:subcrt.est.c2} to exist, $q$ must satisfy, 
in addition to \eqref{l:Kato.est.sub.c1} and \eqref{l:subcrt.est.c1}, 
\begin{equation}\label{t:HH.LWP.sub:pr2}
	\frac1{q} < \frac1{d\alpha} \left(\frac{2\alpha + \gamma}{\alpha-1} -\tilde s\right). 
\end{equation}
Indeed, $\frac{\tilde s+\gamma}{\alpha} < s_c$ 
is equivalent to $\frac1{q} <\frac1{d\alpha} (\frac{2\alpha + \gamma}{\alpha-1} -\tilde s).$ We notice that 
$\frac1{\alpha} (1-\frac{\tilde s}{d}) < \frac1{\alpha} (1-\frac{\gamma}{d(\alpha-1)} )$ 
and 
$\frac1{d\alpha} (\frac{2\alpha + \gamma}{\alpha-1} -\tilde s) > 
\frac1{d} (\frac{2 + \gamma}{\alpha-1} -\tilde s)$ as $\frac{\gamma}{\alpha-1} < \tilde s.$ 
Thus, combining \eqref{l:Kato.est.sub.c1} and \eqref{t:HH.LWP.sub:pr2}, we deduce 
that the conditions for $q$ are \eqref{t:HH.LWP.sub.c1}
\end{proof}

We omit the proofs of $(i),$ $(ii)$ and $(iii)$ of Theorem \ref{t:HH.LWP.sub} 
as they are standard. We only prove $(iv).$
	\subsubsection{Proof of $(iv)$}
Let $u_0 \in L^q_{\tilde s}(\R^d)$ be such that $T_m = T_m(u_0)$ is finite and let 
$u \in C([0,T_m) ; L^q_{\tilde s}(\R^d))$ be the maximal solution of \eqref{HH}. 
Fix $t_0 \in (0,T_m)$ and so that we may express the maximal solution by 
\begin{equation}\nonumber
	u(t+t_0) = e^{t\Delta} u(t_0) + \int_0^t e^{(t-\tau)\Delta} \left\{ |\cdot|^{\gamma} F(u(t_0+\tau,\cdot)) \right\} d\tau, \quad 0\le t < T_m - t_0.
\end{equation}
We observe that 
\begin{equation}\nonumber 
	\|u(t_0)\|_{L^q_{\tilde s}} 
		+ \tilde C_0 (T_m-t_0)^{\frac{\alpha-1}2(s_c-\tilde s)} M^\alpha > M
\end{equation}
holds for all $M>0,$ where $\tilde C_0$ is as in \eqref{l:Kato.est.sub1}.
Otherwise there exists $M>0$ such that 
\begin{equation}\nonumber
	\|u(t_0)\|_{L^q_{\tilde s}} 
		+ \tilde C_0 (T_m-t_0)^{\frac{\alpha-1}2(s_c-\tilde s)} M^\alpha \le M
\end{equation}
so that one may argue as in the proof of existence to obtain a local solution such that 
$\|u(t + t_0)\|_{L^q_{\tilde s}} \le M$ for $t\in [0,T_m -t_0]$ and in particular, 
$u(T_m)$ is well-defined in $L^q_s(\R^d),$ contradicting the definition of $T_m.$ 
Let $M = 2\|u(t_0)\|_{L^q_{\tilde s}}$ so that 
\begin{equation}\nonumber
	\|u(t_0)\|_{L^q_{\tilde s}} + 2^{\alpha} \tilde C_0 \|u(t_0)\|_{L^q_{\tilde s}}^\alpha (T_m-t_0)^{\frac{\alpha-1}2(s_c-\tilde s)}  > 2 \|u(t_0)\|_{L^q_{\tilde s}}, 
\end{equation}
which yields \eqref{t:HH.LWP:Tm}. 
In particular, $\|u(t)\|_{L^q_{\tilde s}} \to \infty$ as $t\to T_m.$ 
Thus, we conclude Theorem \ref{t:HH.LWP.sub}. 

\subsection{Proof of Theorem \ref{t:HH.self.sim}}
Let $\psi (x) := |x|^{-\frac{2+\gamma}{\alpha-1}}$ for $x\ne 0$. 
We first claim that a initial data $u_0$ given by $u_0(x):= c\psi(x)$ with a sufficiently small $c$ satisfies the all assumptions of $(v)$ in Theorem \ref{t:HH.LWP} with $T=\infty,$ thereby generating a global solution to the Cauchy problem (\ref{HH}) with the initial data $u_0$. Since $\psi \in L^1_{loc}(\R^d)$ as $\alpha>\alpha_F(d,\gamma),$ $\psi \in \mathcal{S}'(\R^d)$ and $e^{t\Delta} \psi$ is well-defined. Since $s<s_c,$ there exist $s_1, s_2 \in\R$ such that $s<s_1 < s_c <s_2.$ As in the proof of \cite[Theorem 1.3]{BenTayWei2017}, we can prove that $\psi$ can be decomposed into $\psi = \psi_1 + \psi_2,$ 
$\psi_1 := \chi_{|x|>1} \psi$ and $\psi_2 := \chi_{|x|<1} \psi$ 
so that $\psi_1 \in L^q_{s_1}(\R^d)$ and $\psi_2 \in L^q_{s_2}(\R^d).$ 
This implies that the estimate $\|e^{\Delta} \psi\|_{L^q_s} 
	\le C( \| \psi_1\|_{L^q_{s_1}} + \| \psi_2\|_{L^q_{s_2}})$ holds, 
thanks to Lemma \ref{l:wLpLq}. By the homogeneity of the data, we deduce
$\|e^{t\Delta} \psi\|_{\mathcal{K}^s} < \infty.$ Thus, if the constant $c$ is taken small enough so that $(v)$ in Theorem \ref{t:HH.LWP} is satisfied, the initial data $u_0=c\psi$ generates a unique global solution to (\ref{HH}). 

Let $\varphi := \omega \psi$ be as in the assumption of Theorem \ref{t:HH.self.sim}. 
Then we note that $\varphi$ is homogeneous of degree $-\frac{2+\gamma}{\alpha-1}.$ 
We show that the global solution $u$ to (\ref{HH}) with the initial data 
$\varphi$, which is obtained by $(v)$ in Theorem \ref{t:HH.LWP}, is also self-similar. To this end, for $\lambda>0$, 
let $\varphi_{\lambda}$ be defined by 
$\varphi_{\lambda} (x) := \lambda^{\frac{2-\gamma}{\alpha-1}} \varphi(\lambda x).$  
Since the identity $\|\varphi_{\lambda}\|_{\mathcal{K}^s} = \|\varphi \|_{\mathcal{K}^s}$ holds for all $\lambda>0,$ it follows that $\varphi_{\lambda}$ also 
satisfies the assumptions of $(v)$ in Theorem \ref{t:HH.LWP}. 
As $u_\lambda$ given by \eqref{scale} is a solution of \eqref{HH} with initial data 
$\varphi_{\lambda},$ and $\|u_{\lambda}\|_{\mathcal{K}^s} = \|u \|_{\mathcal{K}^s}$ 
for all $\lambda>0,$ we deduce that $u$ must be self-similar since $\varphi_{\lambda}=\varphi$. We denote the global self-similar solution $u$ by $u_{\mathcal{S}}$. The fact $u_{\mathcal{S}}(t)\rightarrow\varphi$ in $\mathcal{S}'(\R^d)$ as $t\rightarrow +0$ follows from $(iii)$ in Lemma \ref{l:exist.crt}. This completes the proof of Theorem \ref{t:HH.self.sim}. 

\section{Nonexistence of local positive weak solution}
In this section we give a proof of Theorem \ref{t:nonex}. As the argument is standard, we only give a sketch of the proof. 
For the details, we refer to \cite[Proposition 2.4, Theorem 2.5]{II-15}. 
\subsection{Proof of Theorem \ref{t:nonex}}
Let $T\in (0,1)$. Suppose that the conclusion of Theorem \ref{t:nonex} does not hold. 
Then there exists a positive weak solution $u$ on $[0,T)$ 
(See Definition \ref{d:w.sol}). Let 
\[
\psi_T(t,x) := \eta\left(\frac{t}{T}\right) \phi\left(\frac{x}{\sqrt{T}}\right),
\]
where $\eta \in C^\infty_0([0,\infty))$ and $\phi\in C^\infty_0(\mathbb R^d)$ are such that
\[
\eta (t)
:= 
\begin{cases}
1,\quad 0\le t \le \frac12,\\
0,\quad t\ge1,
\end{cases}
\quad \text{and}\quad 
\phi(x)
:= 
\begin{cases}
1,\quad |x| \le \frac12,\\
0,\quad |x|\ge1.
\end{cases}
\]
Let $l\in\mathbb N$ with $l\ge3$, which will be chosen later. 
We note that $\psi_T^l\in C^{1,2}([0,T)\times \R^d)$ and the estimates $|\partial_t \{\psi_T (t,x)\}^l|\le \frac{C}{T} \psi_T(t,x)^{l-1}$ 
and $|\partial_{x_j}^2\{\psi_T(t,x)^l\}|\le \frac{C}{T} \psi_T(t,x)^{l-1}$ hold for $j=1,\ldots, d.$ We define a function $I:[0,T)\rightarrow \R_{\ge 0}$ given by
\[
  I(T):=\int_{[0,T)\times \{|x|<\sqrt{T}\}}|x|^{\gamma} u(t,x)^{\alpha} \, \psi_T^l \, dtdx.
\]
We note that $I(T)<\infty$, since $u\in L_t^{\alpha}(0,T;L^{\alpha}_{\frac{\gamma}{\alpha},loc}(\R^d))$.
By using the weak form (\ref{weak}) and the above estimates, the estimates hold:
\[
\begin{aligned}
I(T) + \int_{|x|<\sqrt{T}} u_0(x) \phi^l\left(\frac{x}{\sqrt{T}}\right)\, dx
& = 
\left|\int_{[0,T)\times \{|x|<\sqrt{T}\}}u(\partial_t \psi_T^l + \Delta \psi_T^l )\,dt\,dx \right|\\
& \le \frac{C}{T}
\int_{[0,T)\times \{|x|<\sqrt{T}\}}|u| \psi_T^{\frac{l}{\alpha}} \,dt\,dx.
\end{aligned}
\]
Here we choose $l$ as 
\begin{equation}\nonumber
	-\frac{l}{\alpha}+l-2>0, \quad\text{i.e.,}\quad l > \frac{2\alpha}{\alpha-1}.
\end{equation}
By H\"older's inequality and Young's inequality, we may estimate the integral 
in the right-hand side above by 
\[
\begin{aligned}
T^{-1}&\int_{[0,T)\times \{|x|<\sqrt{T}\}}|u| \psi_T^{\frac{l}{\alpha}}\, dtdx 
\le I(T)^\frac{1}{\alpha} \cdot T^{-1}K(T)^\frac{1}{\alpha'}
\le \frac12 I(T) + \frac{C}{T^{\alpha'}}K(T).
\end{aligned}
\]
where $1= \frac{1}{\alpha} + \frac{1}{\alpha'}$, i.e., $\alpha'=\frac{\alpha}{\alpha-1}$, and 
\[
K(T) :=  \int_{[0,T)\times \{|x|<\sqrt{T}\}}(|x|^{-\frac{\gamma}{\alpha}})^{\alpha'}\, dtdx
=  T\int_{|x|<\sqrt{T}}|x|^{-\frac{\gamma}{\alpha-1}}\, dx=CT^{1-\frac{\gamma}{2(\alpha-1)}+\frac{d}{2}}
\]
due to $\alpha>1+\gamma/d$. Summarizing the estimates obtained now, we have
\begin{equation}\label{ineq1}
\begin{aligned}
\int_{|x|<\sqrt{T}} u_0(x) \phi^l\left(\frac{x}{\sqrt{T}}\right)\, dx
 \le 
I(T) + 2\int_{|x|<\sqrt{T}} u_0(x) \phi^l\left(\frac{x}{\sqrt{T}}\right)\, dx
 \le 
 C T^{- \frac{2+\gamma}{2(\alpha-1)} + \frac{d}{2}}.
\end{aligned}
\end{equation}
We now choose the initial data $u_0$ as
\[
u_0(x) := 
\begin{cases}
	|x|^{-\beta} \quad & |x|\le 1,\\
	0 & \text{otherwise}
\end{cases}
\]
with 
\begin{equation}\label{beta1}
	\beta< \min\left\{s + \frac{d}{q},d\right\}.
\end{equation}
Then $u_0 \in L^q_s (\R^d)$ and by $T<1$ and $\beta<d$, we have
\begin{equation}\label{ineq2}
\begin{aligned}
	\int_{|x|<\sqrt{T}} u_0(x) \phi^l\left(\frac{x}{\sqrt{T}}\right)\, dx
	& = T^{-\frac{\beta-d}{2}}	\int_{|y|<1} |y|^{-\beta} \phi^l(y)\, dx
	 = C T^{-\frac{\beta-d}{2}}.
\end{aligned}
\end{equation}
Combining \eqref{ineq1} and \eqref{ineq2}, we obtain
\begin{equation}\label{contradiction}
0< C \le T^{\frac{\beta}{2} - \frac{2+\gamma}{2(\alpha-1)}} \to 0 \quad \text{as }T\to 0,
\end{equation}
where
\begin{equation}\label{beta2}
\frac{\beta}{2} - \frac{2+\gamma}{2(\alpha-1)} >0\quad \text{i.e.}\quad \beta > \frac{2+\gamma}{\alpha-1},
\end{equation}
which leads to a contradiction. Thus the proposition holds if we take 
$\beta$ satisfying \eqref{beta1} and \eqref{beta2}, which amount to $s>s_c$ and $\alpha>\alpha_F(d,\gamma)$. The proof is complete.

\section{Appendix}
\par
We list basic properties of the weighted Lebesgue spaces $L^q_s(\R^d)$. 
\begin{prop}
\label{p:wL.sp}
Let $s\in\R$ and $q\in [1,\infty].$ Then the following holds:
\begin{enumerate}[$(1)$]
\item The space $L^q_{s}(\R^d)$ is a Banach space. 
\item $C_0^\infty(\R^d)$ is dense in $L^q_{s}(\R^d)$ if $q$ and $s$ satisfy 
	\begin{equation}\nonumber	
		1\le q < \infty	\quad\text{and}\quad
		-\frac{d}{q} < s < d\left( 1-\frac{1}{q} \right).
	\end{equation}
\item For $s_1, s_2 \in \R,$ $q_1, q_2 \in [1,\infty],$ we have 
	\begin{equation}\nonumber
	\|f\|_{L^q_s} \le \|f\|_{L^{q_1}_{s_1}}^{\theta} \|f\|_{L^{q_2}_{s_2}}^{1-\theta}
	\end{equation}
	for $s = \theta s_1 + (1-\theta) s_2,$ 
	$\frac1{q} = \frac{\theta}{q_1} + \frac{1-\theta}{q_2}$ and $\theta \in (0,1).$ 
\end{enumerate}
\end{prop}
\begin{proof}
$(1)$ The space $L^q_{s}(\R^d)$ is a Lebesgue space 
with a measure $d\mu = |x|^{sq} \,dx.$ See any standard textbook for the proof of 
its completeness. \\
$(2)$ Recall that the weight $|x|^{sq}$ belongs to 
the Muckenhoupt class $A_q$ if and only if $- \frac{d}{q} < s < d(1- \frac1{q})$ 
when $q\in (1,\infty),$ and $|x|^{s} \in A_1$ if and only if $-d<s\le 0$ when $q=1.$ 
Now the density follows from \cite{NakTomYab2004}[Theorem 1.1]. 
\\
$(3)$ For $s$ and $q$ as in the assumption, we have 
\begin{align*}\nonumber
	\|f\|_{L^q_s} 
	&\le \||\cdot|^{s_1} f\|_{L^{q_1}}^{\theta} \||\cdot|^{s_2} f\|_{L^{q_2}}^{1-\theta} 
	= \|f\|_{L^{q_1}_{s_1}}^{\theta}  \|f\|_{L^{q_2}_{s_2}}^{1-\theta}. 
\end{align*}
\end{proof}
The following pointwise bound is well-known in the literature. 
\begin{lem}\label{l:g.unfrm.bnd}
Let $d\in\N,$ $q\in[1,\infty)$ and $a,b,c\in\R.$ 
Let $g(x):=(4\pi)^{-\frac{d}2} e^{-\frac{|x|^2}4}.$ 
\begin{enumerate}[$(1)$]
\item 
There exists a constant $C$ depending only on $d,$ $q,$ $a$ and $b$ such that 
\begin{equation}\nonumber
	\sup_{x\in\R^d} \int_{\R^d} ( |y|^{-a} |x-y|^{b} g(x-y))^q \,dy \le C
\end{equation}
provided that $0\le a<\frac{d}{q}$ and $b\ge0.$ 
\item 
There exists a constant $C$ depending only on $d,$ $q$ and $c$ such that 
\begin{equation}\nonumber
	\sup_{x\in\R^d} \int_{\R^d} ( |y|^{-c} g(x-y))^q \, dy \le C
\end{equation}
provided that $0\le c<\frac{d}{q}.$
\end{enumerate}
\end{lem}
\begin{proof}
In what follows, we shall use the fact that there exists an absolute constant $C$ such that 
\begin{equation}\label{l:g.unfrm.bnd:pr1}
	g(x) \le C \langle x\rangle^{-N}
\end{equation}
for any $N\in\N,$ where $\langle x\rangle :=(1+|x|^2)^{\frac{1}2}.$ Let 
\begin{equation}\nonumber
	\begin{aligned}
	I(x) &:= \int_{\R^d} ( |y|^{-a} |x-y|^{b} g(x-y))^q \, dy \\
	&= \int_{|y|< |x-y|} ( |y|^{-a} |x-y|^{b} g(x-y))^q \, dy 
			+ \int_{|y|> |x-y|} ( |y|^{-a} |x-y|^{b} g(x-y))^q \, dy \\
	&=: I_1 (x) + I_2 (x).
	\end{aligned}
\end{equation}
Thanks to \eqref{l:g.unfrm.bnd:pr1} and $0\le b,$ we have 
\begin{align*}
	I_1(x) \le C \int_{|y|< |x-y|} |y|^{-aq} \langle x-y\rangle^{-(d+1)}\, dy 
	\le C  \int_{|y|< |x-y|} |y|^{-aq} \langle y\rangle^{-(d+1)} \, dy 
	< \infty,
\end{align*}
if $aq<d.$ Moreover, we have 
\begin{equation*}
	I_2(x) \le C \int_{|y|> |x-y|} |x-y|^{-(a-b)q}  g(x-y)^q \, dy 
	\le C \int_{\R^d} |y|^{-(a-b)q}  g(y)^q \, dy < \infty,
\end{equation*}
if $a\ge 0$ and $(a-b)q <d.$ 
Thus, $I(x) < \infty$ uniformly with respect to $x\in\R^d.$ 
The proof for the second inequality is similar so we omit it. 
\end{proof}

We recall the following elementary characterization of $L^1(\R^d)$-functions. 
\begin{prop}	\label{p:L1sg.dcy}
If $f \in L^1(\R^d),$ then 
\begin{equation}\nonumber
	\liminf_{|x|\to 0} |x|^d |f(x)| = \liminf_{|x|\to \infty} |x|^d |f(x)| =0.
\end{equation}
\end{prop}
\begin{proof}
We show the contrapositive. Suppose that 
$\displaystyle\liminf_{|x|\to0} |x|^d |f(x)| = c>0.$ 
Then there exists some positive $\delta$ such that $\frac{c}2 \le |x|^d |f(x)|$ 
for $|x|\le \delta.$ Thus, 
\begin{equation}\nonumber
	\int_{|x|\le\delta} |f(x)| dx \ge c \int_0^\delta r^{-1} \,dr
	= c \left[ \log r\right]_{0}^r = + \infty,
\end{equation}
which implies $f\notin L^1(\R^d).$ The second equality is similarly proved. 
\end{proof}
As a corollary, we have the following. 
\begin{cor}\label{c:wLp.sg.dcy}
Let $s\in\R$ and $p \in [1,\infty].$ 
If $f \in L^{p}_{s}(\R^d),$ then 
\begin{equation}\nonumber
	\liminf_{|x|\to 0} |x|^{s+\frac{d}{p}} |g(x)| 
	= \liminf_{|x|\to \infty} |x|^{s+\frac{d}{p}} |g(x)| =0.
\end{equation}
\end{cor}

Finally, we give a proof of the fact that the $L^q_{\tilde{s}}(\R^d)$-mild solutions 
also satisfy the equation \eqref{HH} in the distributional sense. 
\begin{lem}	\label{mildweak}
We assume the same assumptions as in Theorem \ref{t:HH.LWP} (resp. Theorem \ref{t:HH.LWP.sub}). Let $u$ be a $L^q_{\tilde{s}}(\R^d)$-mild solution on $[0,T)$ in the sense of Definition \ref{def:sol-A}. 
Then $u$ is a weak solution in the sense of Definition \ref{d:w.sol}.
\end{lem}
\begin{proof}
We prove the critical case only, since the subcritical case can be treated in the similar manner. Let $T>0$ and $u$ be an $L^q_{s_c}(\R^d)$-mild solution on $[0,T]$. First we prove $u\in L^{\alpha}(0,T;L^{\alpha}_{\frac{\gamma}{\alpha},loc}(\R^d))$.
Let $\Omega\subset \R^d$ be a compact subset of $\R^d$. We also assume that $q>\alpha$ since the case $q=\alpha$ can be treated in the similar manner with a slight modification. Since $s_c-2\le s<(d+\gamma)/\alpha-d/q$, by the H\"older inequality, the following estimates hold: 
\begin{align*}
   \|u\|_{L^{\alpha}(0,T;L^{\alpha}_{\frac{\gamma}{\alpha}}(\Omega))}^{\alpha}
   &\le \int_0^T\left(\int_{\Omega}|x|^{\frac{q(\gamma-\alpha s)}{q-\alpha}}dx\right)^{\frac{q}{q-\alpha}}\|u(t)\|_{L_s^q} \, dt \\
   &\le C\int_0^Tt^{\frac{s-s_c}{2}}\,dt \, \|u\|_{\mathcal{K}^s(T)}<\infty,
\end{align*}
which implies that $u$ belongs to $L_t^{\alpha}(0,T;L^{\alpha}_{\frac{\gamma}{\alpha},loc}(\R^d))$. Next we prove that $u$ satisfies the weak form (\ref{weak}). Let $\eta\in C^{1,2}([0,T]\times\R^d)$ be such that for any $t\in [0,T]$, $\operatorname{supp} \eta(t, \cdot)$ is compact. Let $T'\in (0,T)$. Since $C_0^{\infty}(\R^d)$ is dense in $L^q_{s_c}(\R^d)$ thanks to Proposition \ref{p:wL.sp}, there exists a sequence $\{u_{0j}\}\subset C_0^{\infty}(\R^d)$ such that the following identity holds:
\[
    \lim_{j\rightarrow\infty}\|u_0-u_{0j}\|_{L^q_{s_c}}=0.
\]
By this identity and the integration by parts, we can prove the following identity:
\begin{align*}
   \int_{[0,T']\times\R^d}&(e^{t\Delta}u_0)(x)(\Delta\eta+\partial_t\eta)(t,x)\,dxdt \\
   &=\int_{\R^d}(e^{T'\Delta}u_0)(x)\, \eta(T',x)\,dx-\int_{\R^d}u_0(x)\eta(0,x)\,dx.
\end{align*}
Thus it suffices to prove the identity
\begin{equation}	\label{weak1}
   \int_{[0,T']\times\R^d}N(u(t,x))(\Delta\eta+\partial_t\eta)(t,x) \,dxdt
   =-\int_{[0,T']\times\R^d}|x|^{\gamma}F(u(t,x))\eta(t,x) \,dxdt,
\end{equation}
where $N$ is defined by (\ref{mapN}). We write $G(t,x):=|x|^{\gamma}F(u(t,x))$. Then we can express $N(u)$ as
\[
N(u)=\int_0^te^{(t-\tau)\Delta}G(\tau)\,d\tau.
\]
Moreover, the equality
\[
    \sup_{t\in [0,T]}t^{\frac{(s_c-s)\alpha}{2}}\|G(t)\|_{L_{\sigma}^{\frac{q}{\alpha}}}=\|u\|_{\mathcal{K}^s(T)}^{\alpha}<\infty
\]
is valid, where $\sigma:=\alpha s-\gamma$. Since the time interval $[0,T]$ is compact, by using mollifiers with respect to the time variable and the space variables, we can find $\{G_j\}\subset C_0^{\infty}([0,\infty)\times \R^d)$ such that 
\begin{equation}	\label{appro1}
    \lim_{j\rightarrow\infty}\sup_{t\in [0,T]}t^{\frac{(s_c-s)\alpha}{2}}\|G(t)-G_j(t)\|_{L^{\frac{q}{\alpha}}_{\sigma}}=0.
\end{equation}
We define a sequence $\{N_j\}$ as
\[
    N_j(t,x):=\int_0^{t}e^{(t-\tau)\Delta}G_j(\tau,x) \, d\tau.
\]
In a similar manner as the proof of Theorem \ref{t:HH.LWP}, we can prove that
\[
    \|N_j-N(u)\|_{\mathcal{K}^s(T)}\le C\sup_{t\in [0,T]}t^{\frac{(s_c-s)\alpha}{2}}\|G_j(t)-G(t)\|_{L_{\sigma}^{\frac{q}{\alpha}}}\rightarrow 0
\]
as $j\rightarrow \infty$. By this fact, we deduce that 
\[
   \text{R.H.S of (\ref{weak1})}
   =\lim_{j\rightarrow\infty}\int_{[0,T']\times\R^d} N_j(t,x)(\Delta\eta+\partial_t\eta)(t,x) \,dxdt.
\]
Since $G_j$ is smooth, so is $N_j$ and hence, by the integration by parts, the identity
\[
   \int_{[0,T']\times\R^d}N_j(t,x)(\Delta\eta+\partial_t\eta)(t,x)\,dx\,dt
   	=\int_{[0,T']\times\R^d}G_j(t,x)\eta(t,x) \,dx\,dt.
\]
holds for any $j$. By taking the limit $j\rightarrow\infty$ in the right-hand side and (\ref{appro1}), we have
\[
   \lim_{j\rightarrow\infty}\int_{[0,T']\times\R^d}
   N_j(t,x)(\Delta\eta+\partial_t\eta)(t,x) \,dxdt
   =\int_{[0,T']\times\R^d}G(t,x)\eta(t,x)\,dxdt.
\]
Thus we obtain (\ref{weak1}), which completes the proof of the lemma.
\end{proof}

\section*{Acknowledgement}
\par
The first author is supported by Grant-in-Aid for Young Scientists (B) 
(No. 17K14216) and Challenging Research (Pioneering) (No.17H06199), 
Japan Society for the Promotion of Science. 
The second author is supported by JST CREST (No. JPMJCR1913), Japan and 
the Grant-in-Aid for Scientific Research (B) (No.18H01132) and 
Young Scientists Research (No.19K14581), JSPS.
The third author is supported by Grant-in-Aid for JSPS Fellows 
(No.19J00206), JSPS.


\end{document}